\documentclass[a4paper]{amsart}

\usepackage[utf8]{inputenc}
\usepackage{amssymb,amsfonts,amsmath}
\usepackage{enumerate}
\usepackage[shortlabels]{enumitem}
\usepackage{color}
\usepackage{mathrsfs}
\usepackage{graphicx}
\usepackage{verbatim}
\usepackage{hyperref}

\newtheorem{theorem}{Theorem}[section]
\newtheorem*{theorem*}{Theorem}
\newtheorem{corollary}[theorem]{Corollary}
\newtheorem*{corollary*}{Corollary}
\newtheorem{lemma}[theorem]{Lemma}
\newtheorem*{lemma*}{Lemma}
\newtheorem{proposition}[theorem]{Proposition}

\newtheorem*{question*}{Question}
\newtheorem{conjecture}[theorem]{Conjecture}
\newtheorem{cit}[theorem]{Citation}

\theoremstyle{definition}
\newtheorem{definition}[theorem]{Definition}
\newtheorem*{definition*}{Definition}

\newtheorem{example}[theorem]{Example}

\newtheorem{notation}[theorem]{Notation}

\theoremstyle{remark}
\newtheorem{remark}[theorem]{Remark}

\newcommand{\Field}{\mathbb{F}}
\newcommand{\N}{\mathbb{N}}

\newcommand{\R}{\mathbb{R}}
\newcommand{\Z}{\mathbb{Z}}

\DeclareMathOperator{\alk}{\lk^\uparrow\!}
\DeclareMathOperator{\asst}{\st^\uparrow\!}
\DeclareMathOperator{\Aut}{Aut}

\DeclareMathOperator{\Ch}{Ch}
\DeclareMathOperator{\conv}{conv}
\DeclareMathOperator{\diam}{diam}
\DeclareMathOperator{\dlk}{\lk^\downarrow\!}
\DeclareMathOperator{\dst}{\st^\downarrow\!}
\DeclareMathOperator{\F}{F}
\DeclareMathOperator{\FP}{FP}

\DeclareMathOperator{\lk}{lk}
\DeclareMathOperator{\op}{op}

\DeclareMathOperator{\Opp}{Opp}

\DeclareMathOperator{\pr}{pr}

\DeclareMathOperator{\pt}{pt}

\DeclareMathOperator{\st}{st}

\DeclareMathOperator{\Sym}{Sym}
\DeclareMathOperator{\thickness}{th}

\newcommand{\abs}[1]{\left\lvert #1 \right\rvert}
\newcommand{\ceil}[1]{\lceil #1 \rceil}
\newcommand{\floor}[1]{\left\lfloor #1 \right\rfloor}
\newcommand\Set[2]{\{\,#1\mid#2\,\}}

\newcommand{\defeq}{\mathrel{\mathop{:}}=}
\newcommand{\eqdef}{=\mathrel{\mathop{:}}}

\begin{document}

\title[Random subcomplexes of buildings, and commutators of RACGs]{Random subcomplexes of finite buildings,\\ and fibering of commutator subgroups of right-angled Coxeter groups}
\date{\today}


\author[E.~Schesler]{Eduard Schesler}
\address{Fakult\"at f\"ur Mathematik und Informatik, FernUniversit\"at in Hagen, 58084 Hagen, Germany}
\email{eduard.schesler@fernuni-hagen.de}

\author[M.~C.~B.~Zaremsky]{Matthew C.~B.~Zaremsky}
\address{Department of Mathematics and Statistics, University at Albany (SUNY), Albany, NY 12222}
\email{mzaremsky@albany.edu}

\begin{abstract}
The main theme of this paper is higher virtual algebraic fibering properties of right-angled Coxeter groups (RACGs), with a special focus on those whose defining flag complex is a finite building.
We prove for particular classes of finite buildings that their random induced subcomplexes have a number of strong properties, most prominently that they are highly connected.
From this we are able to deduce that the commutator subgroup of a RACG, with defining flag complex a finite building of a certain type, admits an epimorphism to $\Z$ whose kernel has strong topological finiteness properties.
We additionally use our techniques to present examples where the kernel is of type $\F_2$ but not $\FP_3$, and examples where the RACG is hyperbolic and the kernel is finitely generated and non-hyperbolic.
The key tool we use is a generalization of an approach due to Jankiewicz--Norin--Wise involving Bestvina--Brady discrete Morse theory applied to the Davis complex of a RACG, together with some probabilistic arguments.
\end{abstract}

\maketitle
\thispagestyle{empty}

\section*{Introduction}

A group $G$ is said to \emph{algebraically fiber} if it admits an epimorphism $G\to\Z$ whose kernel is finitely generated.
If some finite index subgroup of $G$ algebraically fibers, then we say $G$ \emph{virtually algebraically fibers}.
Let us say that a group $G$ \emph{algebraically $F_n$-fibers} if it admits an epimorphism $G\to\Z$ whose kernel is of type $\F_n$.
Recall that a group is of \emph{type $F_n$} if it admits a classifying space with finite $n$-skeleton, so for example type $\F_1$ is equivalent to finite generation and type $\F_2$ is equivalent to finite presentability.
In particular, algebraic $\F_1$-fibering means algebraic fibering.
Note that every group is of type $\F_0$, so a group algebraically $\F_0$-fibers if and only if it admits an epimorphism to $\Z$.
We will generically refer to algebraic $\F_n$-fibering as \emph{higher algebraic fibering}.
Also define \emph{virtually algebraically $F_n$-fibers} and \emph{virtual higher algebraic fibering} in the obvious way.
Recall that ``$\F_\infty$'' is shorthand for ``$\F_n$ for all $n$''.
Finally, recall the homological finiteness properties $\FP_n$; a group is of \emph{type $FP_n$} if the $\Z G$-module $\Z$ admits a projective resolution that is finitely generated in dimensions up to $n$.
We have natural homological analogs to all the above, e.g., \emph{algebraically $FP_n$-fibers}.
It is a standard fact that for $n\ge 2$, $\F_n$ is equivalent to $\F_2$ plus $\FP_n$ (see \cite[Section~VIII.7]{Brown94cohogrps}).

A class of groups for which the question of virtual higher algebraic fibering is a very rich one is the class of right-angled Coxeter groups (RACGs).
\begin{definition*}
A \emph{flag complex} is a simplicial complex in which any finite collection of vertices that pairwise span edges spans a simplex.
Given a finite flag complex $L$, the \emph{right-angled Coxeter group (RACG)} associated to $L$ is the group
\[
W_L\defeq \langle L^{(0)}\mid v^2=1 \text{ for all } v\in L^{(0)} \text{ and } vw=wv \text{ for all } \{v,w\}\in L^{(1)}\rangle \text{.}
\]
\end{definition*}
If we remove the relations $v^2=1$ for all $v$, then we get the \emph{right-angled Artin group (RAAG)} $A_L$ associated to $L$.
For the question of which RAAGs virtually algebraically $\F_n$-fiber, there is a very straightforward sufficient condition: if $L$ is $(n-1)$-connected, then $A_L$ algebraically $\F_n$-fibers \cite{BestvinaBrady97}.
For RACGs, things are much more open.

\begin{question*}
For which flag complexes $L$ does $W_L$ virtually algebraically $\F_n$- and/or $\FP_n$-fiber?
\end{question*}

In this paper, we approach the question of virtual higher algebraic fibering of RACGs from a particular point of view.
Our general approach is heavily influenced by work of Jankiewicz--Norin--Wise \cite{JankiewiczNorinWise21}, who used a combinatorial ``game'' to produce a sufficient condition on $L$ for a RACG $W_L$ to virtually algebraically fiber.
For brevity let us generally refer to this as the \emph{JNW Game}.
Details will be discussed in Subsection~\ref{sec:JNW}.
One nice feature of this approach is that, more than just showing that $W_L$ virtually algebraically fibers, it shows that the commutator subgroup $W_L'$ (which has finite index) algebraically fibers.
In this way the ``virtually'' part of the property becomes constructive, i.e., we get a concrete finite index subgroup.
We remark that, in practice, subgroups of $W_L$ containing the commutator subgroup with smaller index in $W_L$ might also ``work'', so focusing on $W_L'$ is not necessarily optimal, but it is convenient to have a standard description of a finite index subgroup of $W_L$ that works, and it is also interesting that we never have to pass to deeper finite index subgroups than $W_L'$.

We consider the obvious generalization of the JNW Game to higher algebraic fibering, and hence give a sufficient condition for $W_L'$ to algebraically $\F_n$-fiber.
We then present a way to force this sufficient condition to hold, by ensuring that $L$ has ``enough'' highly connected induced subcomplexes, in the following lemma (and see Lemma~\ref{lem:force_legal_system} for a more precise statement).

\begin{lemma*}[Lemma~\ref{lem:force_legal_system}]
If the proportion of induced subcomplexes of $L$ that are not $(m-1)$-connected is less than $(1/2)^{\chi(L)+1}$, where $\chi(L)$ is the chromatic number of $L$, then $W_L$ virtually algebraically $\F_m$-fibers.
\end{lemma*}

The homological version, using $(m-1)$-acyclic and type $\FP_m$, also holds.
We also provide a sufficient condition for the type $\F_m$ (or $\FP_m$) kernel that arises to be not of type $\FP_{m+1}$, see Lemma~\ref{lem:force_sharply_legal_system}.

This approach to virtual higher fibering of RACGs can be seen as an analog to the (non-virtual) situation for RAAGs, initiated by Bestvina--Brady \cite{BestvinaBrady97} and finalized by Meier--Meinert-VanWyk \cite{MeierMeinertVanWyk98} and Bux--Gonzalez \cite{BuxGonzalez99}.
The JNW Game relies on Bestvina--Brady Morse theory, applied to a natural cube complex associated to $W_L$, which reduces the problem to understanding higher connectivity properties of ascending and descending links of vertices.
In the analogous situation for RAAGs in \cite{BestvinaBrady97}, all the vertices have isomorphic ascending links and isomorphic descending links, essentially thanks to the generators having infinite order, and this makes things much more manageable.
For RACGs the ascending and descending links of vertices are constantly changing, essentially thanks to the generators having finite order, and this makes things quite difficult.

After setting up this generalization of the JNW Game, our focus shifts to an interesting class of new examples.
We consider the situation where $L$ is a sufficiently thick finite building, and by analyzing ``random'' induced subcomplexes of certain finite buildings we are able to show using Lemma~\ref{lem:force_legal_system} that certain such $L$ indeed have ``enough'' highly connected induced subcomplexes.

These results on random induced subcomplexes of finite buildings are interesting in their own right.
Buildings are certain simplicial complexes that have an especially high degree of symmetry and connectivity.
They appear in a variety of contexts, e.g., projective geometry, Lie theory, and geometric group theory.
Intuitively, since buildings have such a high degree of symmetry, it stands to reason that a random induced subcomplex of a building ought to retain at least some of the strong connectivity properties of the building.
Our main result in this vein is the following, and see Theorem~\ref{thm:good-homology} for a more precise statement.

\begin{theorem*}[Theorem~\ref{thm:good-homology}]
Let $k \in \N_{\geq 2}$.
Let $\Delta_{k,n}$ be a family of finite buildings of type $A_k$, with thickness going to $\infty$ with $n$.
Then the proportion of induced subcomplexes of $\Delta_{k,n}$ that fail to be $(\floor{\frac{k-1}{2}})$-connected decreases exponentially with $n$.
\end{theorem*}

For example when $k=2$ we get a family of finite simple graphs $\Delta_{2,n}$, with no induced squares, such that for large $n$, ``almost every'' induced subcomplex of $\Delta_{2,n}$ is connected.
We also conjecture (Conjecture~\ref{conj:vanishing-homology-general}) that ``$(\floor{\frac{k-1}{2}})$-connected'' can be improved to ``$(k-2)$-connected'' and the ``type $A_k$'' assumption can be dropped.

The main application to RACGs is that, thanks to Lemma~\ref{lem:force_legal_system}, this produces natural examples of RACGs that virtually algebraically $\F_{\floor{\frac{k+1}{2}}}$-fiber (and conjecturally $\F_{k-1}$-fiber), in Theorem~\ref{thrm:typeA_higher_fibering}:

\begin{theorem*}[Theorem~\ref{thrm:typeA_higher_fibering}]
Let $k\in\N_{\ge 2}$. For all but finitely many $n$, the group $W_{\Delta_{k,n}}'$ algebraically $\F_{\floor{\frac{k+1}{2}}}$-fibers.
\end{theorem*}

We also produce some examples involving negative finiteness properties, in Theorem~\ref{thrm:low_dim_sharp_fiber} (in particular kernels that are of type $\F_2$ but not $\FP_3$), and examples of hyperbolic RACGs that virtually algebraically fiber with non-hyperbolic kernel, in Theorem~\ref{thrm:typeA_graph_fibering}.

As a first step towards Theorem~\ref{thm:good-homology}, we prove in Section~\ref{sec:random-subcomplexes-are-chamber} that random induced subcomplexes of a finite building $\Delta$ are chamber complexes, provided $\Delta$ is thick enough.
Here a chamber complex is a finite dimensional simplicial complex that roughly resembles a building, in that its maximal simplices (called chambers) all have the same dimension and such that any two chambers can be connected by a gallery.
Besides being chamber complexes, we also show that random induced subcomplexes of a thick enough finite building $\Delta$ also inherit a further property of $\Delta$, namely that they are unions of apartments.
A simplified version of this result can be stated as follows:

\begin{theorem*}[Theorem~\ref{thm:random-subcomplexes-of-arbitrary-sequences}]
Let $d \in \N$ and let $(\Delta_n)_{n \in \N}$ be a sequence of finite, $d$-dimensional Moufang buildings with thickness going to $\infty$ with $n$.
Suppose that for each $n$ every panel of $\Delta_n$ is contained in the same number of chambers.
Then the proportion of induced subcomplexes of $\Delta_{n}$ that are $d$-dimensional chamber complexes that are unions of apartments tends to $1$ as $n \rightarrow \infty$.
\end{theorem*}

We emphasize that, unlike Theorem~\ref{thm:good-homology}, Theorem~\ref{thm:random-subcomplexes-of-arbitrary-sequences} has no restriction on the (spherical) type of the $\Delta_n$.

One important step in proving Theorem~\ref{thm:random-subcomplexes-of-arbitrary-sequences} is establishing some ``flexibility'' for choosing a chamber with prescribed projection images onto a given set of panels.
The crucial observation is that this flexibility can be encoded with the help of higher dimensional analogs of magic squares (see Definition~\ref{def:magic-cube}).
This encoding then allows us to obtain the needed flexibility via an upper bound on the side lengths of a block of $0$-entries in a magic square (see Lemma~\ref{lem:zero-block} for a precise statement).

\medskip

We remark that some notions of randomness have already produced results in conjunction with the JNW Game.
For example, in \cite[Section~8]{JankiewiczNorinWise21} it is shown that in some sense for random $L$, $W_L'$ algebraically fibers.
More precisely, given a fixed number of vertices $n$, build a simple graph by including each potential edge independently with probability $p$, and then for $L$ the flag complex on this graph and $p$ in an appropriate range, $W_L'$ algebraically fibers (see \cite[Theorem~8.4]{JankiewiczNorinWise21}, and see \cite{FizPontiverosGlebovKarpas17} for a better bound on $p$).
Here we do not use random $L$, but rather have a fixed $L$ and in some sense inspect random induced subcomplexes of $L$.

For another idea of the usefulness of the JNW Game, very recently, Italiano--Martelli--Migliorini used it to find the first known example of a hyperbolic group that algebraically $\F_\infty$-fibers with the kernel not hyperbolic itself \cite{ItalianoMartelliMigliorini21}.
(In fact the kernel is \emph{type $F$}, meaning it has a finite classifying space.)
This was a major open problem, and such examples had previously only been found with $\F_\infty$ replaced by $\F_2$ \cite{Brady99,Lodha18,KrophollerVigolo21}.
In particular this gives the first example of a non-hyperbolic type $\F_\infty$ (even type $\F$) group with no Baumslag--Solitar subgroups.
It remains open whether for $n>2$ there exists a hyperbolic group that algebraically $\F_n$-fibers with a kernel that is not of type $\FP_{n+1}$, and the JNW Game seems like a promising avenue for resolving this in the future as well.

Finally, let us remark that the question of algebraic $\F_n$-fibering of a group $G$ of type $\F_n$ is connected to the Bieri--Neumann--Strebel--Renz (BNSR) invariants $\Sigma^n(G)$, developed in \cite{BieriNeumannStrebel87,BieriRenz88}.
More precisely, if $G$ algebraically $\F_n$-fibers then $\Sigma^n(G)$ is non-empty, and while the converse is not literally true, it is true that if $\Sigma^n(G)$ contains a pair of antipodal (rational) points then $G$ algebraically $\F_n$-fibers.
We will not define the BNSR-invariants here, since they will only come up briefly, but see \cite[Section~8]{Bux04} for all the relevant background and details.

This paper is organized as follows.
In Sections~\ref{sec:RACG} and~\ref{sec:morse} we recall some background material on RACGs and discrete Morse theory, respectively.
Section~\ref{sec:vaf_RACG} is the main section on virtual higher algebraic fibering, in which we recall the JNW Game, generalize it, and establish the key Lemma~\ref{lem:force_legal_system}.
In Section~\ref{sec:buildings} we recall some background on buildings.
Section~\ref{sec:magic-squares} is devoted to developing our generalization of magic squares, which leads to important independence results about projections to panels and chambers.
After collecting some technical results about certain sequences in Section~\ref{sec:calc}, we prove in Section~\ref{sec:random-subcomplexes-are-chamber} our results about random induced subcomplexes of certain finite buildings being chamber complexes.
Finally, in Section~\ref{sec:buildings-An} we prove our main results about finite buildings of type $A_n$.
Section~\ref{sec:graphs} is a quick but interesting observation about consequences in dimension $1$, and in Section~\ref{sec:apps} we prove our main results about virtual higher algebraic fibering of RACGs.

\subsection*{Acknowledgments} The authors are grateful to Mikhail Ershov, Dawid Kielak, Rob Kropholler, Kevin Schreve, and Stefan Witzel for a number of helpful discussions.
Thanks are also due to the anonymous referee for many helpful suggestions.
The first author was partially supported by the DFG grant WI 4079/4 within the SPP 2026 Geometry at Infinity.
The second author is supported by grant \#635763 from the Simons Foundation.

\section{Background on right-angled Coxeter groups}\label{sec:RACG}

Let $L$ be a finite flag complex and $W_L$ its associated RACG.
Recall that by an \emph{induced subcomplex} of $L$ we mean a subcomplex $X \leq L$ where a simplex $\sigma \subseteq L$ is contained in $X$ if and only if $X$ contains the vertex set of $\sigma$.
If $X$ is an induced subcomplex of $L$, then $W_X$ naturally embeds as a subgroup of $W_L$, called a \emph{standard parabolic subgroup}.
An \emph{induced square} in $L$ is an induced subgraph of $L^{(1)}$ that is a square.
Call $L$ \emph{square-free} if it has no induced squares.
Note that the RACG associated to a square is $D_\infty \times D_\infty$, which contains $\Z^2$.
In particular if $L$ has induced squares then $W_L$ is not (Gromov) hyperbolic.
The converse holds as well, by a result of Moussong \cite{Moussong88}, and so we have:

\begin{cit}\cite{Moussong88}\label{cit:hyp_RACG}
The RACG $W_L$ is hyperbolic if and only if $L$ is square-free.
\end{cit}

Every RACG $W_L$ has an associated CAT(0) cube complex $X_L$ called the \emph{Davis complex}, constructed as follows.
First we have $X_L^{(0)}=W_L$, and for every $g\in W_L$ and every simplex $\sigma$ in $L$, the $0$-cubes $g\prod_{v\in\tau^{(0)}}v$ for each $\tau\le \sigma$ span a cube in $X_L$.
(Here we include $\tau=\emptyset$.)
For example, if $\sigma$ is a $1$-simplex with vertex set $\{v,w\}$, then there is a $2$-cube in $X_L$ with set of $0$-cubes $\{g,gv,gw,gvw\}$.
It is clear that $X_L$ is simply connected, since the $2$-cubes correspond to the defining relations of $W_L$, and that the link of every vertex is isomorphic to $L$, hence is flag, so $X_L$ is CAT(0), hence contractible.

The action of $W_L$ on itself by left translation extends to a cubical action of $W_L$ on $X_L$.
This action is transitive and free on $0$-cubes, and $X_L$ is locally compact, so the action of $W_L$ on $X_L$ is \emph{geometric}, that is, proper (meaning cube stabilizers are finite) and cocompact (meaning the orbit space is compact).
In particular the action of any finite index subgroup of $W_L$ on $X_L$ is also geometric.

It is easy to see that the abelianization $W_L^{ab}$ is finite, namely $(\Z/2\Z)^{\abs{L^{(0)}}}$, so the commutator subgroup $W_L'$ has finite index in $W_L$.
The commutator subgroup $W_L'$ consists of all elements of $W_L$ with an even number of every generator.
The geometric action of $W_L'$ on $X_L$ will be of particular interest in what follows.
Note that the stabilizers in $W_L$ of cubes in $X_L$ are the conjugates of the standard parabolic subgroups generated by vertex sets of simplices in $L$.
Clearly $W_L'$ intersects all of these trivially, so the action of $W_L'$ on $X_L$ is free, and hence $W_L'$ is torsion-free.

\subsection{Virtual algebraic fibering of RACGs}\label{sec:vaf_racgs}

The question of which RACGs virtually algebraically fiber is clarified to some extent by a result of Kielak.
In \cite[Theorem~5.3]{Kielak20}, he proves that for an infinite finitely generated group $G$ that is virtually RFRS (residually finite rationally solvable), $G$ virtually algebraically fibers if and only if its first $L^2$-Betti number $\beta_1^{(2)}(G)$ is zero.
For example, a right-angled Artin group $A_L$ (which is indeed virtually RFRS \cite{Agol08}) whose defining flag complex $L$ is connected has $\beta_1^{(2)}(A_L)=0$ \cite[Corollary~2]{DavisLeary03}, and hence virtually algebraically fibers.
(Indeed $A_L$ already algebraically fibers, as the kernel of the map $A_L\to\Z$ sending every generator to $1$, called the Bestvina--Brady subgroup, is finitely generated \cite{BestvinaBrady97}.)

RACGs are also virtually RFRS \cite{Agol08}, so \cite[Theorem~5.3]{Kielak20} applies and we see that $W_L$ virtually algebraically fibers if and only if $\beta_1^{(2)}(W_L)=0$.
Less is known though about $\beta_1^{(2)}$ for RACGs than for RAAGs.
As an example of something that is known, if $L$ is a triangulation of an $n$-sphere ($n\ge 2$) then $\beta_1^{(2)}(W_L)=0$ \cite[Theorem~11.3.2]{DavisOkun01} and so $W_L$ virtually algebraically fibers.

In \cite{JankiewiczNorinWise21}, Jankiewicz--Norin--Wise set up the JNW Game, which we will discuss in Subsection~\ref{sec:JNW}, and use it to prove virtual algebraic fibering for a variety of examples of RACGs.
For example in \cite[Section~5]{JankiewiczNorinWise21} they prove this for $L$ equal to the $1$-skeleton of a $3$-cube, the $1$-skeleton of an icosahedron, and many other examples with ``lots of'' edges.
Throughout, they relate virtual algebraic fibering to the ``Charney--Davis $n$-curvature'' $\kappa_n(L)$ (see \cite[Subsection~3b]{JankiewiczNorinWise21}, and also \cite{CharneyDavis95EulerChar}).
This is defined by
\[
\kappa_n(L)\defeq \sum_{k=-1}^n (-1/2)^{k+1} \ell_k\text{,}
\]
where $\ell_k$ is the number of $k$-simplices of $L$ (so $\ell_{-1}=1$).
(In \cite[Subsection~3b]{JankiewiczNorinWise21} there is a typo: $(-2)^{k+1}$ should be $(-1/2)^{k+1}$.)
In particular
\[
\kappa_2(L)=1-\frac{\ell_0}{2}+\frac{\ell_1}{4} \text{.}
\]
One notable application of $\kappa_2$ is that if the JNW Game, which we will discuss in Subsection~\ref{sec:JNW}, succeeds and reveals that $W_L$ virtually algebraically fibers, then necessarily $\kappa_2(L)\ge 0$ \cite[Theorem~6.14]{JankiewiczNorinWise21}.
Also, in this case if moreover $\kappa_2(L)=0$ then $W_L$ virtually algebraically $\F_\infty$-fibers.

As the above indicates, if $\kappa_2(L)<0$ then $W_L$ does not virtually algebraically fiber.
Also just generally speaking, it is easier to find situations where $W_L$ does not virtually algebraically fiber.
Let us discuss some more such situations.

\begin{lemma}\label{lem:disconn_no_fiber}
If $L$ is disconnected then $W_L$ does not virtually algebraically fiber, unless $L=S^0$ so $W_L\cong D_\infty$.
\end{lemma}

\begin{proof}
Suppose $W_L$ does virtually algebraically fiber, so $\beta_1^{(2)}(W_L)=0$.
By \cite[Theorem~7.3.3]{DavisOkun01}, for disconnected $L$ this can only happen if $L$ is a disjoint union of simplices.
Hence $W_L$ is virtually free.
In particular $W_L$ cannot virtually algebraically fiber, unless $L=S^0$ so $W_L\cong D_\infty$.
\end{proof}

If $L$ is a planar graph with no induced cycles of length less than $6$, then \cite[Theorem~5]{KarNikolov14} shows that $\beta_1^{(2)}(W_L)=\kappa_2(L)$.
Combining this with \cite[Theorem~5.3]{Kielak20} and some calculations, we get the following reasonably strong restriction on virtual algebraic fibering of RACGs:

\begin{lemma}\label{lem:planar_no_fiber}
If $L$ is a planar graph with no cycles of length less than $6$, then $W_L$ does not virtually algebraically fiber, unless $L$ is a path of length $2$ (so $W_L\cong D_\infty\times \Z/2\Z$) or $L=S^0$ (so $W_L\cong D_\infty$).
\end{lemma}

\begin{proof}
By Lemma~\ref{lem:disconn_no_fiber}, without loss of generality $L$ is connected.
By \cite[Theorem~5]{KarNikolov14} and \cite[Theorem~5.3]{Kielak20}, $W_L$ virtually algebraically fibers if and only if $1-\frac{\ell_0}{2}+\frac{\ell_1}{4}=0$.
If $L$ is a tree then $\ell_1=\ell_0-1$, so $0=1-\frac{\ell_0}{2}+\frac{\ell_0-1}{4}$, i.e., $\ell_0=3$, which means $L$ is a path of length $2$.
Now assume $L$ is not a tree.
Since $L$ is connected and has no cycles of length less than $6$, by \cite[Theorem~1.5.3]{Jungnickel13}, $\ell_1\le \frac{3}{2}(\ell_0-2)$.
Now $0=1-\ell_0/2+\ell_1/4 \le 1-\ell_0/2+3(\ell_0-2)/8 = (2-\ell_0)/8$, so $\ell_0\le 2$, which is impossible.
\end{proof}

In particular we see that for $L$ a tree, other than a path of length $2$, $W_L$ does not virtually algebraically fiber.
This is in stark contrast to RAAGs, where for $L$ a tree, $A_L$ algebraically $\F_\infty$-fibers \cite{BestvinaBrady97}.

\section{Background on discrete Morse theory}\label{sec:morse}

In this section we recall some details about discrete Morse theory that we will need later.
Let $Y$ be an affine cell complex, in the sense of \cite{BestvinaBrady97}, for example a simplicial or cubical complex.
A \emph{Morse function} is a map
\[
h\colon Y\to\R
\]
that restricts to an affine map on cells, is non-constant on each positive dimensional cell, and such that the image $h(Y^{(0)})$ of the vertex set is closed and discrete in $\R$.
(In \cite{BestvinaBrady97} it only says ``discrete'', but it is clear that it needs to be closed as well.)
Thanks to the hypotheses, we see that every cell has a unique vertex at which $h$ is maximized, and a unique vertex at which $h$ is minimized.
The \emph{ascending star} $\asst v$ of a vertex $v$ is the subcomplex of $Y$ consisting of all cells with $v$ as their vertex with minimum $h$ value, and their faces.
The \emph{ascending link} $\alk v$ of $v$ is the link of $v$ in $\asst v$, that is, the space of directions out of $v$ along which $h$ increases.
Analogously define the \emph{descending star} $\dst v$ and \emph{descending link} $\dlk v$.

The point of discrete Morse theory is that understanding the ascending and descending links of vertices can lead to an understanding of much larger, globally defined spaces, namely sublevel and superlevel sets.
For each $t\in\R$ let
\[
Y^{h\le t}\defeq h^{-1}((-\infty,t]) \text{ and } Y^{h\ge t}\defeq h^{-1}([t,\infty))
\]
be the \emph{sublevel set} and \emph{superlevel set} of $Y$ relative $h$ at level $t$.
Note that as $t$ varies, each of these forms a \emph{filtration} of $Y$, meaning a nested sequence of subspaces whose union is the whole space.
The following Morse Lemma relates the topology of the ascending/descending links to that of the super/sublevel sets, see, e.g., \cite[Corollary~2.6]{BestvinaBrady97}.

\begin{cit}[Morse Lemma]\label{cit:morse}
Let $h\colon Y\to\R$ be a Morse function.
Let $t\le s$ (allowing for $t=-\infty$ and $s=\infty$).
If $\alk v$ is $(m-1)$-acyclic (resp. $(m-1)$-connected) for all vertices $v$ with $t\le h(v)<s$, then the inclusion $Y^{h\ge s}\to Y^{h\ge t}$ induces an isomorphism in $\widetilde{H}_k$ (resp. $\pi_k$) for all $k\le m-1$ and a surjection in $\widetilde{H}_m$.
If $\dlk v$ is $(m-1)$-acyclic (resp. $(m-1)$-connected) for all vertices $v$ with $t<h(v)\le s$, then the inclusion $Y^{h\le t}\to Y^{h\le s}$ induces an isomorphism in $\widetilde{H}_k$ (resp. $\pi_k$) for all $k\le m-1$ and a surjection in $\widetilde{H}_m$.
\end{cit}

For example, if $s=\infty$ and $Y$ itself is, say, contractible, then as soon as all the $\alk v$ for $h(v)<t$ are $(m-1)$-acyclic or $(m-1)$-connected, so is $Y^{h\ge t}$.

Discrete Morse theory is especially powerful when coupled with Brown's Criterion, and has become a standard tool for deducing finiteness properties of groups.
Let us recall Brown's Criterion here, in the degree of generality we need.

\begin{cit}[Brown's Criterion]\cite{Brown87finprops}\label{cit:brown}
Let $Y$ be an $(m-1)$-acyclic (resp. $(m-1)$-connected) complex on which a group $G$ acts properly.
Let $\{Y_t\}_{t\in D}$ be a filtration of $Y$ indexed by a directed set $D$, so $Y$ is the union of the $Y_t$, and for $t\le s$ in $D$ we have $Y_t\subseteq Y_s$.
Suppose each $Y_t$ is $G$-invariant and cocompact.
Then $G$ is of type $\FP_m$ (resp. type $\F_m$) if and only if the filtration $\{Y_t\}_{t\in D}$ is essentially $(m-1)$-acyclic (resp. essentially $(m-1)$-connected).
\end{cit}

Here \emph{essentially $(m-1)$-acyclic} means for all $t\in D$ there exists $s\ge t$ such that the inclusion $Y_t\to Y_s$ induces the trivial map in $\widetilde{H}_k$ for all $k\le m-1$.
Similarly, \emph{essentially $(m-1)$-connected} is defined using $\pi_k$.

Discrete Morse theory is clearly useful for proving that a filtration coming from a Morse function is essentially $(m-1)$-acyclic or $(m-1)$-connected (and in fact usually proves something stronger, that the pieces of the filtration themselves are already $(m-1)$-acyclic or $(m-1)$-connected, bypassing the word ``essentially'').
In practice, the sublevel and superlevel sets relative to a Morse function may not be $G$-cocompact, but one could instead filter using $h^{-1}([-t,t])$, or use additional techniques from the world of BNSR-invariants.
We will not go into more detail here on all this, but will spell it out as it comes up later.

Let us also discuss a technique for proving that a filtration coming from a Morse function is not essentially $m$-acyclic, which is useful for proving that a group is not of type $\FP_m$.
The following is a purely homological version of \cite[Proposition~2.6]{Zaremsky17sepPB}, and the proof is essentially identical.
The proof uses \cite[Corollary~2.4]{Zaremsky17sepPB}, which was already purely homological.
Note that \cite{Zaremsky17sepPB} uses a different notion of ``Morse function'', but the definition we use here is a special case of that one, so everything from \cite{Zaremsky17sepPB} still applies.

\begin{proposition}\label{prop:not_ess_acyc}
Let $Y$ be an affine cell complex and $h\colon Y\to\R$ a Morse function.
Assume $\widetilde{H}_{m+1}(Y)=0$.
Suppose there exists $N\in\R$ such that for all vertices $v\in Y^{(0)}$ with $h(v)<N$ the ascending link $\alk v$ is $(m-1)$-acyclic and satisfies $\widetilde{H}_{m+1}(\alk v)=0$.
Assume moreover that for all $M\in\R$ there exists a vertex $v\in Y^{(0)}$ with $h(v)<M$ such that $\widetilde{H}_m(\alk v)\ne 0$.
Then the filtration $\{Y^{h\ge t}\}_{t\in\R}$ is not essentially $m$-acyclic.
\end{proposition}

\begin{proof}
Suppose $\{Y^{h\ge t}\}_{t\in\R}$ is essentially $m$-acyclic.
Take some $t<N$, and choose $-\infty<s\le t$ such that the inclusion $Y^{h\ge t}\to Y^{h\ge s}$ induces $0$ in $\widetilde{H}_k$ for all $k\le m$.
Since $t<N$, the Morse Lemma also tells us that this inclusion induces a surjection in all these $\widetilde{H}_k$, so in fact $Y^{h\ge s}$ is $m$-acyclic.
In fact all $Y^{h\ge r}$ are $m$-acyclic for $r\le s$.
Choose $v\in Y^{(0)}$ with $h(v)<s$ and $\widetilde{H}_m(\alk v)\ne 0$.
Since every $Y^{h\ge r}$ is $m$-acyclic for $r\le s$, Mayer--Vietoris plus \cite[Corollary~2.4]{Zaremsky17sepPB} tell us that $\widetilde{H}_{m+1}(Y^{h\ge q})\ne 0$ for all $q\le h(v)$.
This includes $q=-\infty$, which is to say that $\widetilde{H}_{m+1}(Y)\ne 0$, a contradiction.
\end{proof}

This improves \cite[Proposition~2.6]{Zaremsky17sepPB} in two ways: our assumption using $(m-1)$-acyclic is weaker than the assumption there using $(m-1)$-connected, and our conclusion of, ``not essentially $m$-acyclic,'' is stronger than the conclusion there of, ``not essentially $m$-connected.''

\section{Higher algebraic fibering of commutator subgroups of right-angled Coxeter groups}\label{sec:vaf_RACG}

Let us recall the JNW Game developed by Jankiewicz--Norin--Wise in \cite{JankiewiczNorinWise21}, phrased in our current language, and extended to considerations of higher connectivity.

\subsection{The JNW Game}\label{sec:JNW}

Let $L$ be a finite flag complex.
Call a subset $\sigma\subseteq L^{(0)}$ of vertices a \emph{state}.
A state is \emph{$k$-legal} if the subcomplexes of $L$ induced by $\sigma$ and $L^{(0)}\setminus \sigma$ are both $k$-connected.
In particular $0$-legal is what is called ``legal'' in \cite{JankiewiczNorinWise21}.
For each vertex $v\in L^{(0)}$, choose a subset $\mu_v\subseteq L^{(0)}$, called a \emph{move}, satisfying $v\in \mu_v$ and for any $w$ adjacent to $v$, $w\not\in \mu_v$.
A \emph{system of moves} $M$ is a choice of a move for each vertex.
Note that it could happen that different vertices are assigned the same move.
Perhaps the most straightforward example of a system of moves is a \emph{colored system}: given a partition of $L^{(0)}$ into blocks such that adjacent vertices never share a block, the blocks form a systems of moves.

Now we make the following key observation: The power set $\mathcal{P}(L^{(0)})$ of the generating set $L^{(0)}$ of $W_L$ is in bijection with the quotient group $W_L/W_L'$.
Indeed, the latter is the finite abelian group $\mathcal{A}\defeq (\Z/2\Z)^{\abs{L^{(0)}}}$, which is obviously identifiable with $\mathcal{P}(L^{(0)})$.
This identification provides a group structure to $\mathcal{P}(L^{(0)})$, for example $\emptyset$ corresponds to the identity and the symmetric difference corresponds to the group operation.
Now any state or move is an element of $\mathcal{A}$, and any system of moves is a subset of $\mathcal{A}$.
Choose a system of moves $M=\{\mu_v\mid v\in L^{(0)}\}$, and let $\mathcal{M}$ be the subgroup of $\mathcal{A}$ generated by $M$.
Call the system of moves $M$ \emph{$k$-legal} if there exists a coset in $\mathcal{A}/\mathcal{M}$ whose elements are all $k$-legal, and call this a \emph{$k$-legal coset}.

The following is the higher algebraic fibering analog of Theorem~4.3 and Corollary~4.4 of \cite{JankiewiczNorinWise21}, and the proof is essentially the same.
We will take this opportunity to flesh out many of the details that were implicit in the proof in \cite{JankiewiczNorinWise21}.

\begin{proposition}\label{prop:JNW_game}
If $L$ admits an $(m-1)$-legal system of moves, then $W_L'$ algebraically $\F_m$-fibers.
\end{proposition}

\begin{proof}
Let $M$ be an $(m-1)$-legal system of moves, and let $\mathcal{M}$ be the subgroup of $\mathcal{A}$ generated by $M$.
Let $\sigma+\mathcal{M}$ be an $(m-1)$-legal coset of $\mathcal{M}$, so in particular $\sigma$ is an $(m-1)$-legal state.
We now construct a height function $h$ on the $0$-skeleton of $X_L$.
Assign the $0$-cube of $X_L$ corresponding to the identity element $1$ of $W_L$ a height of $h(1)=0$.
The $0$-cubes of $X_L$ adjacent to $1$ correspond to the generators of $W_L$, i.e., the vertices of $L$.
Any such vertex either lies in $\sigma$ or $L^{(0)}\setminus \sigma$.
If a given vertex $v$ of $L$ lies in $\sigma$, declare that the $0$-cube $v$ of $X_L$ has height $h(v)=1$.
If it lies in $L^{(0)}\setminus \sigma$, declare it has height $h(v)=-1$.
For a given such $v$, the $0$-cubes of $X_L$ adjacent to $v$ are of the form $vw$ for $w$ a vertex of $L$.
We assign heights to these $vw$ by inspecting whether $w$ lies in the ($(m-1)$-legal) state $\sigma+\mu_v$ or $L^{(0)}\setminus (\sigma+\mu_v)$.
(Intuitively, now that we are centered at $v$, we apply the move $\mu_v$ to recalibrate the state.)
If $w$ lies in $\sigma+\mu_v$, declare that the height of $vw$ is $h(vw)=h(v)+1$, and if $w$ lies in $L^{(0)}\setminus (\sigma+\mu_v)$, declare that $h(vw)=h(v)-1$.
Note that if $w=v$ then $vw=1$, so we need to make sure $h$ is well defined, but it is clear that $v\in \sigma$ if and only if $v\not\in \sigma+\mu_v$, so indeed we get $h(vv)=0$.
Continuing in this way, we assign heights to every $0$-cube of $X_L$.
To ensure well definedness we need to make sure that whenever $v$ and $w$ are adjacent vertices of $L$, and so $vw=wv$, a well defined height is assigned to $gvw=gwv$ for any $0$-cube $g$.
Indeed, this is clear from the definition of move, since no vertex can lie in the move of an adjacent vertex.
In other words, $h(gwv)-h(gw) = h(gv)-h(g)$ for any $w$ adjacent to $v$, since $v\not\in \mu_w$.

By now we have a function $h\colon X_L^{(0)}\to\Z$.
By construction, this function extends to a function $h\colon X_L\to\R$ that is affine on cubes.
It is clearly a Morse function.
For any $0$-cube $g$, the ascending link of $g$ is isomorphic to the induced subcomplex of $L$ spanned by all vertices in some state in the coset $\sigma+\mathcal{M}$, and the descending link is isomorphic to the subcomplex induced by the complement of this state.
Since this coset is $(m-1)$-legal, the ascending and descending links are all $(m-1)$-connected.
By the Morse Lemma (Citation~\ref{cit:morse}), this shows that for all $t\in\R$ the preimages $h^{-1}([t,\infty))$ and $h^{-1}((-\infty,t])$ are $(m-1)$-connected.

Now consider $W_L'$, viewed as a subset of $X_L^{(0)}$.
For any sequence of generators $v_1,\dots,v_k$, if $v_1\cdots v_k \in W_L'$ then every generator appears an even number of times, and so $\mu_{v_1}+\cdots+\mu_{v_k}=0$.
In particular for any $g\in W_L'$ and any $x\in W_L=X_L^{(0)}$, we have $h(gx)=h(g)+h(x)$.
This shows that $h$ restricted to $W_L'$ is a homomorphism $\psi\colon W_L'\to\Z$, and moreover the action of $W_L'$ on $X_L$ is $\psi$-equivariant.
This action is also proper and cocompact, since the same is true of $W_L$ and $W_L'$ has finite index in $W_L$.
Now we can finish the proof by appealing to BNSR-invariants.
Since $h^{-1}([t,\infty))$ and $h^{-1}((-\infty,t])$ are $(m-1)$-connected, we know that $[\pm\psi]\in\Sigma^m(W_L')$ by \cite[Definition~8.1]{Bux04}, and so the kernel of $\psi$ is of type $\F_m$ by \cite[Citation~8.4]{Bux04}.
\end{proof}

\begin{remark}\label{rmk:homological}
All of the above can be done homologically instead of homotopically.
We get a notion of \emph{homologically $(m-1)$-legal}, where we want the induced subcomplexes to be $(m-1)$-acyclic instead of $(m-1)$-connected, and this leads to $W_L'$ algebraically $\FP_m$-fibering.
The proof works analogously.
\end{remark}

\begin{remark}
Of course the point of Proposition~\ref{prop:JNW_game} is that it works for large $m$, but it is interesting to point out what happens when $m=0$.
In this case, as soon as $L$ is not a simplex, we claim there is a $(-1)$-legal system of moves.
Indeed, let $v$ and $w$ be non-adjacent vertices of $L$, set $\mu_v=\mu_w=\{v,w\}$, and for all $u\in L^{(0)}\setminus \{v,w\}$ set $\mu_u=\{u\}$.
Now the state $\{v\}$ is $(-1)$-legal, i.e., non-empty and with non-empty complement, and applying any combination of moves results in a state that either contains $v$ and not $w$, or $w$ and not $v$, hence is also $(-1)$-legal.
Thus this system of moves is $(-1)$-legal.
Confirming Proposition~\ref{prop:JNW_game}, for any $W_L$ with $L$ not a simplex, there is an epimorphism $W_L\to D_\infty$ and hence an epimorphism $W_L'\to\Z$.
\end{remark}

\subsection{Ensuring an $(m-1)$-legal system}\label{sec:force_legal_system}

Now we discuss a situation in which we can ensure that $L$ admits an $(m-1)$-legal system of moves.
The key is to get ``enough'' vertices of $L$ assigned to the same move, so as to make $M$ have small rank, and thus make the number of cosets of $\mathcal{M}$ in $\mathcal{A}$ large.
Then if ``most'' induced subcomplexes of $L$ are $(m-1)$-connected, the pigeonhole principle (roughly with non-$(m-1)$-connected induced subcomplexes being pigeons, and cosets of $\mathcal{M}$ in $\mathcal{A}$ being pigeonholes) will do the rest.

An easy way to construct a system of moves is using colorings.
The \emph{chromatic number} $\chi(L)$ of $L$ is the chromatic number of $L^{(1)}$, that is, the smallest number $\chi(L)\in\N$ such that there exists a function $c\colon L^{(0)}\to\{1,\dots,\chi(L)\}$ with no adjacent vertices mapping to the same element, called a \emph{coloring}.

\begin{definition}\label{def:bad_conn}
For each $k\ge-1$, let $\mathcal{F}_k(L)$ denote the set of induced subcomplexes $X$ of $L$ such that $X$ is not $k$-connected.
\end{definition}

\begin{lemma}\label{lem:force_legal_system}
If $\abs{\mathcal{F}_{m-1}(L)} < 2^{\abs{L^{(0)}}-\chi(L)-1}$ then $L$ admits an $(m-1)$-legal system of moves, and so $W_L'$ algebraically $\F_m$-fibers.
\end{lemma}

\begin{proof}
Let $c\colon L^{(0)}\to\{1,\dots,\chi(L)\}$ be a coloring.
For each vertex $v$, let $\mu_v\defeq c^{-1}(c(v))$, so all the vertices of the same ``color'' as $v$.
These $\mu_v$ clearly form a (colored) system of moves $M$.
The subgroup $\mathcal{M}$ generated by $M$ has rank $\chi(L)$, hence order $2^{\chi(L)}$, so the number of cosets of $\mathcal{M}$ in $\mathcal{A}$ is $2^{\abs{L^{(0)}}-\chi(L)}$.
Since $\abs{\mathcal{F}_{m-1}(L)} < 2^{\abs{L^{(0)}}-\chi(L)-1}$, by the pigeonhole principle there must exist an $(m-1)$-legal coset (recall that for a state to be $(m-1)$-legal we need both it and its complement to induce an $(m-1)$-connected subcomplex, whence the extra factor of $2$).
Hence our system of moves is $(m-1)$-legal.
\end{proof}

\begin{example}\label{ex:bipartite}
As an example of the $m=1$ case, let $L$ be a bipartite graph, say with $n$ vertices, such that of the $2^n$ induced subgraphs of $L$, strictly fewer than $2^n/8$ are disconnected.
Then Lemma~\ref{lem:force_legal_system} says that $W_L'$ algebraically fibers.
\end{example}

We can also define
\[
\mathcal{F}_k^{hom}(L)
\]
to be the set of induced subcomplexes of $L$ that are not $k$-acyclic, and get a homological version of Lemma~\ref{lem:force_legal_system}: If $\abs{\mathcal{F}_{m-1}^{hom}(L)} < 2^{\abs{L^{(0)}}-\chi(L)-1}$ then $L$ admits a homologically $(m-1)$-legal system of moves, and so $W_L'$ algebraically $\FP_m$-fibers.

Note that finding the $(m-1)$-legal system of moves was more or less constructive, since it came directly from a choice of coloring of the vertices of $L$.
However, finding the specific $(m-1)$-legal coset was not constructive, and so in particular identifying an explicit character $W_L'\to\Z$ with kernel of type $\F_m$, the existence of which is guaranteed by Proposition~\ref{prop:JNW_game}, would be difficult.
There are a few things one could say about the resulting character, for example its kernel necessarily contains every simple commutator $vwvw$ of vertices $v$ and $w$ with different colors.

\begin{corollary}\label{cor:almost_all_fiber}
Let $(L_n)_{n\in\N}$ be a family of finite flag complexes, all with the same chromatic number $\chi$.
Suppose the function 
\[
\N \rightarrow \R_{\geq 0},\ n \mapsto \frac{\abs{\mathcal{F}_{m-1}(L_n)}}{2^{\abs{L_n^{(0)}}}}
\]
goes to $0$ as $n$ goes to $\infty$.
Then for all but finitely many $n$, $L_n$ admits an $(m-1)$-legal system of moves, so $W_{L_n}'$ algebraically $\F_m$-fibers.
\end{corollary}

\begin{proof}
Let $n$ be large enough that $\frac{\abs{\mathcal{F}_{m-1}(L_n)}}{2^{\abs{L_n^{(0)}}}} < \frac{1}{2^{\chi+1}}$.
This means $\abs{\mathcal{F}_{m-1}(L_n)} < 2^{\abs{L_n^{(0)}}-\chi-1}$, so Lemma~\ref{lem:force_legal_system} says $L_n$ admits an $(m-1)$-legal system of moves, and $W_{L_n}'$ algebraically $\F_m$-fibers.
\end{proof}

Note that the sequence $\chi(L_n)$ need not be constant for this proof to work.
Indeed, we just need that $\frac{\abs{\mathcal{F}_{m-1}(L_n)}}{2^{\abs{L_n^{(0)}}}} < \frac{1}{2^{\chi(L_n)+1}}$ for all but finitely many $n$.
For example, if $\frac{\abs{\mathcal{F}_{m-1}(L_n)}}{2^{\abs{L_n^{(0)}}}}$ is exponentially decreasing in $n$, and $\chi(L_n)$ is sublinear in $n$, then the same result would hold.
We also note once again that the analogous homological version of Corollary~\ref{cor:almost_all_fiber} holds, using $\mathcal{F}_{m-1}^{hom}(L_n)$, homologically $(m-1)$-legal, and algebraically $\FP_m$-fibers.

\subsection{Negative finiteness properties}\label{sec:negative}

Let us now discuss a situation where we can say that $W_L'$ algebraically $\F_m$-fibers with a map $W_L'\to\Z$ whose kernel is not only of type $\F_m$ but specifically not of type $\FP_{m+1}$.

\begin{definition}[Sharply $k$-legal]
Call a state $\sigma\subseteq L^{(0)}$ \emph{sharply $k$-legal} if it is $k$-legal, so the subcomplexes of $L$ induced by $\sigma$ and $L^{(0)}\setminus \sigma$ are both $k$-connected, and moreover neither of these subcomplexes are $(k+1)$-acyclic, and moreover they both have trivial $(k+2)$nd homology.
If a system of moves generates a subgroup of $\mathcal{A}$ with a coset whose elements are all sharply $k$-legal, we will also call the system of moves and any such coset \emph{sharply $k$-legal}.
Also, we add the adverb ``homologically'' if we want to consider $k$-acyclic instead of $k$-connected.
\end{definition}

For example if the induced subcomplexes in question are all $(k+1)$-spheres, then the state is sharply $k$-legal.

\begin{lemma}\label{lem:JNW_game_with_negative}
If $L$ admits a sharply $(m-1)$-legal system of moves, then there is a map $W_L'\to\Z$ whose kernel is of type $\F_m$ but not $\FP_{m+1}$.
\end{lemma}

\begin{proof}
Returning to the proof of Proposition~\ref{prop:JNW_game} and all the notation therein, we have a map $\psi\colon W_L'\to\Z$, and its kernel is of type $\F_m$ by virtue of $h^{-1}([t,\infty))$ and $h^{-1}((-\infty,t])$ being $(m-1)$-connected for all $t$.
Now to see that the kernel is not of type $\FP_{m+1}$, it suffices to prove that the filtration $\{h^{-1}([t,\infty))\}_{t\in\R}$ of $X_L$ is not essentially $m$-acyclic, since then $[\psi]\not\in\Sigma^{m+1}(W_L';\Z)$.
We will actually show that neither $\{h^{-1}([t,\infty))\}_{t\in\R}$ nor $\{h^{-1}((-\infty,t])\}_{t\in\R}$ is essentially $m$-acyclic.
Since the system of moves is sharply $(m-1)$-legal, the ascending and descending links are not only $(m-1)$-connected (hence $(m-1)$-acyclic), but also have non-trivial $\widetilde{H}_m$ and trivial $\widetilde{H}_{m+1}$.
Hence Proposition~\ref{prop:not_ess_acyc} says that neither $\{h^{-1}([t,\infty))\}_{t\in\R}$ nor $\{h^{-1}((-\infty,t])\}_{t\in\R}$ is essentially $m$-acyclic.
\end{proof}

Analogously, if $L$ admits a sharply homologically $(m-1)$-legal system of moves, then there is a map $W_L'\to\Z$ whose kernel is of type $\FP_m$ but not $\FP_{m+1}$.

\begin{definition}\label{def:top_hlgy}
For each $k\ge 0$, let $\mathcal{T}_k(L)$ denote the set of induced subcomplexes of $L$ that have trivial $k$th reduced homology.
\end{definition}

\begin{lemma}\label{lem:force_sharply_legal_system}
Suppose $L$ has dimension $d$.
If $\abs{\mathcal{F}_{d-1}(L)} + \abs{\mathcal{T}_d(L)} < 2^{\abs{L^{(0)}}-\chi(L)-1}$ then $L$ admits a sharply $(d-1)$-legal system of moves, and so $W_L'$ algebraically $\F_d$-fibers with a map $W_L'\to\Z$ whose kernel is not of type $\FP_{d+1}$.
\end{lemma}

\begin{proof}
This follows by an analogous proof to Lemma~\ref{lem:force_legal_system}.
We see that there must exist a coset of $\mathcal{M}$ in $\mathcal{A}$, every element of which is a state such that it and its complement both induce subcomplexes of $L$ that are both $(d-1)$-connected and have non-trivial $d$th reduced homology.
Since $L$ is $d$-dimensional, no subcomplex can have non-trivial $(d+1)$st homology, so the system of moves generating $\mathcal{M}$ is sharply $(d-1)$-legal.
The rest now follows from Lemma~\ref{lem:JNW_game_with_negative}.
\end{proof}

The homological version works as well, by an analogous proof: If $\abs{\mathcal{F}_{d-1}^{hom}(L)} + \abs{\mathcal{T}_d(L)} < 2^{\abs{L^{(0)}}-\chi(L)-1}$ then $L$ admits a sharply homologically $(d-1)$-legal system of moves, and so $W_L'$ algebraically $\FP_d$-fibers with a map $W_L'\to\Z$ whose kernel is not of type $\FP_{d+1}$.

\begin{corollary}\label{cor:almost_all_sharply_fiber}
Let $(L_n)_{n\in\N}$ be a family of finite flag complexes, all with the same chromatic number $\chi$ and dimension $d$.
Suppose the function
\[
\N \rightarrow \R_{\geq 0},\ n \mapsto \frac{\abs{\mathcal{F}_{d-1}(L_n)}+\abs{\mathcal{T}_d(L_n)}}{2^{\abs{L_n^{(0)}}}}
\]
goes to $0$ as $n$ goes to $\infty$.
Then for all but finitely many $n$, $L_n$ admits a sharply $(d-1)$-legal system of moves, so $W_{L_n}'$ algebraically $\F_d$-fibers with a map $W_L'\to\Z$ whose kernel is not of type $\FP_{d+1}$.
\end{corollary}

\begin{proof}
Let $n$ be large enough that $\frac{\abs{\mathcal{F}_{d-1}(L_n)}+\abs{\mathcal{T}_d(L_n)}}{2^{\abs{L_n^{(0)}}}} < \frac{1}{2^{\chi+1}}$.
This means that $\abs{\mathcal{F}_{d-1}(L_n)} + \abs{\mathcal{T}_d(L)} < 2^{\abs{L_n^{(0)}}-\chi-1}$, so Lemma~\ref{lem:force_sharply_legal_system} gives us the result.
\end{proof}

Once again the homological version holds too, by an analogous proof, using $\mathcal{F}_{d-1}^{hom}(L_n)$, homologically sharply $(d-1)$-legal, and algebraically $\FP_d$-fibers.

\section{Background on buildings}\label{sec:buildings}

The rest of this paper will deal with random subcomplexes of finite (spherical) buildings.
Let us start by recalling the definition of a building (see, e.g.,~\cite[Definition 4.1]{AbramenkoBrown08}).

\begin{definition}[Building]\label{def:building}
A simplicial complex $\Delta$ is called a \emph{building} if there is a set $\mathcal{A}$ consisting of subcomplexes $\Sigma \leq \Delta$, the so-called \emph{apartments}, that satisfy the following conditions
\begin{enumerate}
\item[(B0)] Each apartment $\Sigma \in \mathcal{A}$ is isomorphic to the Coxeter complex $\Sigma(W,S)$ of some Coxeter system $(W,S)$.
\item[(B1)] Every two simplices $A,B \subseteq \Delta$ are contained in some apartment $\Sigma \in \mathcal{A}$.
\item[(B2)] For every two apartments $\Sigma_1,\Sigma_2 \in \mathcal{A}$ there is an isomorphism $\Sigma_1 \rightarrow \Sigma_2$ fixing $\Sigma_1 \cap \Sigma_2$ pointwise.
\end{enumerate}
\end{definition}

In view of (B2), all the apartments must have the same Coxeter type, so any building $\Delta$ also has a well defined Coxeter type $(W,S)$.
If each apartment is finite, we say that $\Delta$ is a \emph{spherical} building.
Most of the buildings we consider are finite, and therefore spherical.
In general, all Coxeter complexes are finite-dimensional, so it makes sense to consider the set $\Ch(\Delta)$ of \emph{chambers}, i.e., maximal simplices, in $\Delta$.
More generally, if $X$ is a subcomplex of $\Delta$, we will write $\Ch(X) \subseteq \Ch(\Delta)$ to denote the set of chambers that are contained in $X$.
A simplex $P < \Delta$ of dimension $\dim(P) = \dim(\Delta)-1$ is called a \emph{panel}.
From (B1) it follows that for every two chambers $C,D \in \Ch(\Delta)$ there is a sequence $\Gamma = (E_i)_{i=1}^n$ of chambers $E_i \in \Ch(\Delta)$ with $E_1 = C$ and $E_n = D$ such that $E_i \cap E_{i+1}$ is a panel for $1 \leq i < n$.
In this case we write $\Gamma = E_1\ |\ E_2\ |\ \ldots\ |\ E_n$ and call $\Gamma$ a \emph{gallery} from $C$ to $D$.
More generally, if $A$ is a face of $C$ we say that $\Gamma$ is a gallery from $A$ to $D$.

\begin{definition}[Convex]\label{def:convex-hull}
Let $\Delta$ be a building and let $X \subseteq \Delta$ be a subcomplex that is a union of chambers of $\Delta$.
We say that $X$ is \emph{convex} if for all $C,D \in \Ch(X)$ every minimal gallery from $C$ to $D$ stays in $X$.
Given a subset $S \subseteq \Ch(\Delta)$ we define the \emph{convex hull} $\conv(S) \subseteq \Delta$ to be the smallest convex subcomplex containing $S$.
\end{definition}

We will use the following well-known result about the convex hull of two chambers.

\begin{lemma}\label{lem:conv-hull-of-two-chambers}
Let $C,D$ be two chambers in a building $\Delta$.
The convex hull $\conv(C,D)$ coincides with the union of all chambers that lie on a minimal gallery from $C$ to $D$.
\end{lemma}
\begin{proof}
If $\Delta$ is a Coxeter complex the lemma is an easy exercise as can be found in~\cite[1.66]{AbramenkoBrown08}.
Suppose now that $\Delta$ is an arbitrary building.
By (B1) there is an apartment $\Sigma \subseteq \Delta$ containing $C$ and $D$.
Since every apartment in a building is convex (see, e.g.,~\cite[4.40]{AbramenkoBrown08}), we see that $\conv(C,D) \subseteq \Sigma$.
Now the lemma follows from the case of Coxeter complexes.
\end{proof}

Given two chambers $C,D$ in a building $\Delta$, we define their \emph{gallery distance} $d_{\Delta}(C,D)$ to be the length of a minimal gallery from $C$ to $D$.
Suppose that $\Delta$ is spherical.
Then it follows from (B1) that the \emph{diameter} of $\Delta$, given by
\[
\diam(\Delta) \defeq \max \Set{d_{\Delta}(C,D)}{C,D \in \Ch(\Delta)},
\]
is finite.
From the convexity of apartments in buildings (see, e.g.,~\cite[4.40]{AbramenkoBrown08}) we further obtain $d_{\Delta}(C,D) = d_{\Sigma}(C,D)$ whenever $\Sigma$ is an apartment containing $C$ and $D$.
The case where $d_{\Delta}(C,D) = \diam(\Delta)$ will be of special interest for us.
If $C$ is a chamber in an apartment $\Sigma$ of $\Delta$, then there is a unique chamber $D \subset \Sigma$ with $d_{\Delta}(C,D) = \diam(\Delta)$ (see, e.g.,~\cite[1.57]{AbramenkoBrown08}).
In view of this, the following definition makes sense.

\begin{definition}[Opposition]\label{def:opposition}
Let $C$ be a chamber in a spherical building $\Delta$.
A chamber $D \in \Ch(\Delta)$ is called \emph{opposite} to $C$ in $\Delta$ if $d_{\Delta}(C,D) = \diam(\Delta)$.
For each apartment $\Sigma \subseteq \Delta$ with $C \subset \Sigma$ we define $\op_{\Sigma}(C)$ to be the unique chamber in $\Sigma$ that is opposite to $C$.
The set of all chambers in $\Delta$ that are opposite to $C$ will be denoted by $\Opp_{\Delta}(C)$.
More generally, if $\mathcal{E} \subseteq \Ch(\Delta)$ is a subset, we define
\[
\Opp_{\Delta}(\mathcal{E}) \defeq \bigcap \limits_{C \in \mathcal{E}} \Opp_{\Delta}(C)
\]
as the set of chambers in $\Delta$ that are opposite to every chamber in $\mathcal{E}$.
\end{definition}

\smallskip

\noindent A major theme later will be that random subcomplexes of a finite building $\Delta$ tend to be highly connected, assuming the so-called thickness of $\Delta$ is large compared to the dimension of $\Delta$.

\begin{definition}[Thickness]\label{def:thickness}
Let $\Delta$ be a building.
The \emph{thickness} of $\Delta$, denoted by $\thickness(\Delta)$, is the minimal number $t$ such that every panel of $\Delta$ is contained in at least $t$ chambers.
If $\thickness(\Delta) \geq 3$, we say that $\Delta$ is a \emph{thick} building.
In the special case where every panel of a thick building $\Delta$ is contained in exactly $\thickness(\Delta)$ chambers, we say that $\Delta$ is \emph{uniformly thick}.
\end{definition}

Buildings of type $A_n$ are the easiest to describe.
Such buildings will be the subject of Section~\ref{sec:buildings-An}, where we prove some results on higher connectivity properties of random subcomplexes of $A_n$-buildings.

\begin{example}\label{ex:An_building}
Let $d \in \N$ and let $V$ be a $(d+1)$-dimensional vector space over a field $\Field$.
Consider the graph $\Gamma$ whose vertex set consists of the non-trivial proper subspaces of $V$ and where two vertices $U,W$ are connected by an edge if either $U \subset W$ or $W \subset U$.
Let $A(V)$ denote the flag complex of $\Gamma$, i.e., the simplicial complex whose $1$-skeleton is given by $\Gamma$ and where a finite set $\sigma$ of vertices of $\Gamma$ spans a simplex in $A(V)$ if and only if every $2$-element subset of $\sigma$ spans an edge in $\Gamma$.
The complex $A(V)$ is a building of type $A_{d}$ and its dimension is given by $d-1$ (see, e.g.,~\cite[Section 4.2]{AbramenkoBrown08}).
\end{example}

We will make use of the following notions of (relative) links and stars.

\begin{definition}[(Relative) star/link]\label{def:link-and-star}
Let $X$ be a simplicial complex.
For every simplex $\sigma \subseteq X$ we define the \emph{star} of $\sigma$ in $X$, denoted by $\st_X(\sigma)$, to be the subcomplex of $X$ consisting of all simplices $\tau \subseteq X$ joinable to $\sigma$.
Further we define the \emph{link} of $\sigma$ in $X$, denoted by $\lk_X(\sigma)$, to be the subcomplex of $X$ consisting of simplices $\tau \subseteq \st_X(\sigma)$ with $\sigma \cap \tau = \emptyset$.
More generally, if $Y \leq X$ is a subcomplex and $\sigma \subseteq X$ is a simplex, we define the \emph{relative star} $\st_Y(\sigma) \defeq \st_X(\sigma) \cap Y$ and the \emph{relative link} $\lk_Y(\sigma) \defeq \lk_X(\sigma) \cap Y$ of $\sigma$ in $Y$.
\end{definition}

A function $\N\to\N$ is called \emph{polynomially bounded} if it is bounded above by a polynomial.

\begin{lemma}\label{lem:polynomial-bounded}
For every $d \in \N$ there is a polynomially bounded function $q_d \colon \N \rightarrow \N$ with the following property.
If $\Delta$ is a finite, thick, $d$-dimensional building $\Delta$, such that every panel $P$ of $\Delta$ is contained in at most $t$ chambers, then the number of cells in $\Delta$ is bounded above by $q_d(t)$.
\end{lemma}
\begin{proof}
Let $\Delta$ be a finite $d$-dimensional building of type $(W,S)$.
Recall that the maximal gallery distance in $\Delta$ is given by \[
\diam(\Delta) = \diam(\Sigma(W,S)) \eqdef \ell_{W,S}.
\]
Since every chamber $C \subset \Delta$ contains exactly $d+1$ panels and every panel is contained in at most $t$ chambers it follows that there are at most $(d+1)(t-1)$ chambers $D \subset \Delta$ that share a panel with $C$.
Thus for every $k \in \N$ there are at most $(d+1)^k(t-1)^k$ galleries of length $k$ in $\Delta$ that start with $C$.
Since every chamber $D \subset \Delta$ can be reached by a gallery of length at most $\ell_{W,S}$ starting from $C$ we see that the number of chambers in $\Delta$ is bounded above by $q_{W,S} \defeq \sum \limits_{k=0}^{\ell_{W,S}} (d+1)^k(t-1)^k$. 
Note that $q_{W,S}$ is a polynomial in $t$ that only depends on the type of $\Delta$.

To prove the lemma it remains to apply the well-known result of Feit and Higman~\cite{FeitHigman64} that every connected component of the Coxeter diagram of a finite thick building is of type $A_n$, $C_n$, $D_n$, $E_n$, $F_4$, $G_2$, or $I_2(8)$.
From this it follows that for every $d \in \N$ there are only finitely many Coxeter types $X$ of finite thick $d$-dimensional buildings, so that we can define $q_d$ as the pointwise maximum of the polynomials $q_X$.
\end{proof}

Let $\Aut(\Delta)$ denote the group of type preserving automorphisms of $\Delta$.
To ensure that $\Aut(\Delta)$ acts ``transitively enough'' for our purposes, we will often work with buildings that satisfy the \emph{Moufang} property.
This seemingly strong restriction was proven by Tits~\cite[Satz 1]{Tits77} to hold for every irreducible, thick, spherical building of dimension at least $2$.
To define the Moufang property for spherical buildings, we have to recall the notion of a \emph{root} $R$ in a spherical Coxeter complex $\Sigma$.
Informally, $R$ can be characterized as a subcomplex of $\Sigma$ that lies on one side of one of the hyperplanes that give rise to the cell structure of $\Sigma$.
We can therefore think of roots as subcomplexes that look like hemispheres.
More formally, a subcomplex $R \subset \Sigma$ is a root if there is a chamber $C \in \Ch(\Sigma)$ and a panel $P \subset C$ such that $R$ is the union of all chambers $E \in \Ch(\Sigma)$ that can be reached by a minimal gallery from $P$ to $E$ that starts with $C$.

\begin{definition}[Moufang]\label{def:Moufang}
Let $\Delta$ be a spherical building, let $\Sigma \subseteq \Delta$ be an apartment, and let $R \subset \Sigma$ be a root.
We say that a panel $P \subset R$ is an \emph{interior panel} of $R$ if it is not contained in the boundary $\partial R$.
The \emph{root group} corresponding to $R$, denoted by $U_R$, consists of the automorphisms $\alpha \in \Aut(\Delta)$ that fix the star $\st_{\Delta}(P)$ of every interior panel $P \subset R$ pointwise.
The building $\Delta$ is called \emph{Moufang} if all root groups $U_{R}$ act transitively on the sets of the form $\Ch(\st_{\Delta}(P)) \setminus \{C\}$, where $P \subset \partial R$ is a panel and $C$ is the unique chamber in $R$ that has $P$ as a face.
\end{definition}

\section{Independent projections}\label{sec:magic-squares}

Recall that a magic square is a square matrix consisting of non-negative integers such that the sums of integers in each row and column coincide.
In this section we set up our ``magic cubes'' construction, which introduces a higher dimensional, probabilistic generalization of magic squares.
From the viewpoint of ordinary magic squares our first goal is to prove that certain distributions of $0$-entries cannot arise if the sum in each row and column is positive.
Next we observe that for a finite building $\Delta$, every $n$-element set $\{P_1,\ldots,P_n\}$ of panels in $\Delta$ gives rise to an $n$-dimensional magic cube whose entries are parametrized by the set $\prod_{i=1}^{n} \Ch(\st(P_i))$ of sequences of chambers in the stars of these panels.
In this situation, the $0$-entries of the magic cube correspond to sequences of chambers $(E_i)_{i=1}^{n}$ that cannot arise as a sequence of projection images of the form $(\pr_{P_i}(C))_{i=1}^{n}$, where $C$ is an arbitrary chamber in $\Delta$.
Using our observations on general magic cubes, we will then prove Theorem~\ref{thm:disjoint-unions-of-apartments}, which provides us with the existence of certain sets of chambers in $\Delta$ that will play a key role in Section~\ref{sec:random-subcomplexes-are-chamber}.

\subsection{Magic cubes}

\begin{definition}[Magic cube]\label{def:magic-cube}
Let $n \in \N$ and let $X$ be a set.
A map $\mu \colon \mathcal{P}(X^n) \rightarrow \N_0$ is called an \emph{$n$-dimensional magic cube over $X$} if there is a number $N$, the \emph{weight of $\mu$}, such that
\begin{enumerate}
\item $\mu$ is a measure on $(X^n,\mathcal{P}(X^n))$.
\item For every $x \in X$ and every $1 \leq i \leq n$ the preimage of $x$ under the projection $\pi_i\colon X^n\to X$ to the $i$th coordinate satisfies $\mu(\pi_i^{-1}(x)) = N$.
\end{enumerate}
We say that a subset $A \subseteq X^n$ has positive weight if $\mu(A) > 0$.
In the case of a singleton $A = \{a\}$ we write $\mu(a) = \mu(A)$ and say that $a$ has positive weight if $\mu(a) > 0$.
\end{definition}

The following example provides us with an easy way of constructing magic cubes.

\begin{example}\label{ex:magic-cube}
Let $n \in \N$ and let $X,Y$ be finite non-empty sets.
For each $1 \leq i \leq n$, let $f_i \colon Y \rightarrow X$ be a map such that the fibers $f_i^{-1}(x)$ have the same cardinality $N$ for all $1 \leq i \leq n$ and $x\in X$.
Note that in this case $N = \abs{Y} \cdot \abs{X}^{-1}$.
Consider the map
\[
f \colon Y \rightarrow X^n,\ y \mapsto (f_1(y),\ldots,f_n(y)).
\]
For every $1 \leq i \leq n$ and every $x\in X$ we have $\abs{f^{-1}(\pi_i^{-1}(x))} = \abs{f_i^{-1}(x)} = N$.
Thus we see that
\[
\mu \colon \mathcal{P}(X^n) \rightarrow \N_0,\ A \mapsto \abs{f^{-1}(A)}
\]
is an $n$-dimensional magic cube of weight $N$ over $X$.
\end{example}

For a set $X$ let $\Sym(X)$ denote the group of permutations on $X$.

\begin{remark}\label{rem:permutation-invariance}
Let $\mu$ be an $n$-dimensional magic cube of weight $N$ over a finite set $X$, let $\sigma_i \in \Sym(X)$ for $1 \leq i \leq n$, and let $\sigma \defeq \sigma_1 \times \ldots \times \sigma_n \in \Sym(X^n)$ be the product of these permutations.
Consider the pushforward measure $\sigma_{\ast}(\mu)$ of $\mu$ with respect to $\sigma$, i.e., the measure on $X^n$ given by $\sigma_{\ast}(\mu)(A) = \mu(\sigma^{-1}(A))$ for all $A \subseteq X^n$.
Note that $\sigma_{\ast}(\mu)$ satisfies
\[
\sigma_{\ast}(\mu)(\pi_i^{-1}(x)) = \mu(\pi_i^{-1}(\sigma_i^{-1}(x))) = N
\]
for every $x \in X$ and every $1 \leq i \leq n$.
Thus $\sigma_{\ast}(\mu)$ is an $n$-dimensional magic cube of weight $N$ over $X$ as well.
\end{remark}

\begin{lemma}\label{lem:zero-block}
Let $\mu$ be an $n$-dimensional magic cube of weight $N > 0$ over $X \defeq \{1,\ldots,t\}$.
If there is some $1 \leq k \leq t$ with $\mu(\{1,\ldots,k\}^n)=0$, then $\frac{k}{t} < \frac{n^2}{1+n^2}$.
\end{lemma}
\begin{proof}
We consider the slices $S_i \defeq \{i\} \times X^{n-1}$ for $1\leq i \leq k$.
If $x = (i,x_2,\ldots,x_n) \in S_i$ has positive weight, then there is at least one coordinate $2 \leq j \leq n$ with $x_j > k$.
From our assumption we know that $\mu(S_i) = N > 0$.
Thus we can find a number $f(i) \in \{2,\ldots,n\}$ with
\[
\mu(S_i \cap \pi_{f(i)}^{-1}(\{k+1,\ldots,t\})) > \frac{N}{n}.
\]
This gives us a function $f \colon \{1,\ldots,k\} \rightarrow \{2,\ldots,n\}$.
Let $j \in \{2,\ldots,n\}$ be such that the cardinality of $f^{-1}(j)$ is maximal.
In this case we have $\abs{f^{-1}(j)} > \frac{k}{n}$.
Thus there are $k_0 \defeq \ceil{\frac{k}{n}}$ distinct elements $1 \leq i_1 < \ldots < i_{k_0} \leq k$ with
\[
\mu(S_{i_l} \cap \pi_{j}^{-1}(\{k+1,\ldots,t\})) > \frac{N}{n}
\]
for $1 \leq l \leq k_0$.
Since the slices $S_{i_l}$ are pairwise disjoint we obtain
\begin{align*}
\mu(\pi_{j}^{-1}(\{k+1,\ldots,t\}))
&\geq \mu\left(\left(\bigcup \limits_{l=1}^{k_0} S_{i_l}\right) \cap \pi_{j}^{-1}(\{k+1,\ldots,t\})\right)\\
&= \mu\left(\bigcup \limits_{l=1}^{k_0} (S_{i_l} \cap \pi_{j}^{-1}(\{k+1,\ldots,t\}))\right)\\
&= \sum \limits_{l=1}^{k_0} \mu(S_{i_l} \cap \pi_{j}^{-1}(\{k+1,\ldots,t\}))\\
&> \sum \limits_{l=1}^{k_0} \frac{N}{n}\\
&\geq \frac{kN}{n^2}.
\end{align*}
On the other hand, the weight of $\pi_{j}^{-1}(\{k+1,\ldots,t\})$ is given by
\[
\mu(\pi_{j}^{-1}(\{k+1,\ldots,t\}))
= \sum \limits_{i=k+1}^{t} \mu(\pi_{j}^{-1}(i))
= \sum \limits_{i=k+1}^{t} N
= (t-k) \cdot N.
\]
Hence we get the inequality $\frac{k N}{n^2} < (t-k) \cdot N$ from which it is easy to see that $\frac{k}{t} < \frac{n^2}{1+n^2}$.
\end{proof}

\begin{corollary}\label{cor:zero-block}
Let $\mu$ be an $n$-dimensional magic cube of weight $N > 0$ over $X \defeq \{1,\ldots,t\}$.
Then there are permutations $\sigma_i \in \Sym(X)$ for $1 \leq i \leq n$ such that $(\sigma_1(i),\ldots,\sigma_n(i))$ has positive weight for all $1 \leq i \leq \ceil{\frac{t}{1+n^2}}$.
\end{corollary}
\begin{proof}
Let $k$ be maximal with the property that we can find $\sigma_1,\ldots,\sigma_n \in \Sym(X)$
such that $(\sigma_1(i),\ldots,\sigma_n(i))$ has positive weight for all $1 \leq i \leq k$.
In view of Remark~\ref{rem:permutation-invariance} we may assume that each $\sigma_j$ is the identity on $X$.
Suppose that there is an element $x = (x_1,\ldots,x_n) \in \{k+1,\ldots,t\}^n$ of positive weight and let $\tau_1,\ldots,\tau_n \in \Sym(\{1,\ldots,t\})$ be the transpositions where $\tau_i$ interchanges $x_i$ with $k+1$.
Then we have
\[
\mu(\tau_1(k+1),\ldots,\tau_n(k+1)) = \mu(x_1,\ldots,x_n) > 0 \text{,}
\]
and for any $1\leq i\leq k$,
\[
\mu(\tau_1(i),\ldots,\tau_n(i)) = \mu(i,\ldots,i) > 0 \text{.}
\]
But this is a contradiction to the maximality of $k$.
Thus, we see that $\mu(\{k+1,\ldots,t\}^n) = 0$.
By a further application of Remark~\ref{rem:permutation-invariance} we can replace $\{k+1,\ldots,t\}^n$ with $\{1,\ldots,t-k\}^n$, so that $\mu(\{1,\ldots,t-k\}^n) = 0$.
In this case Lemma~\ref{lem:zero-block} tells us that $\frac{t-k}{t} < \frac{n^2}{1+n^2}$, from which it can be easily derived that $k > \frac{t}{1+n^2}$.
\end{proof}

\subsection{Independent projections to panels}

We fix a finite, uniformly thick Moufang building $\Delta$ of thickness $t \defeq \thickness(\Delta)$ and a sequence $P_1,\ldots,P_n$ of distinct panels in $\Delta$.
For every $P_i$ let $\{C_{i,1},\ldots,C_{i,t}\}$ be an enumeration of the chambers in $\st(P_i)$.

\smallskip

It can be shown (see, e.g., in~\cite[Proposition 4.95]{AbramenkoBrown08}) that for every simplex $A \subseteq \Delta$ and every chamber $C \subseteq \Delta$ there is a unique chamber $E \subseteq \st_{\Delta}(A)$ such that every minimal gallery from $A$ to $C$ starts with $E$.
The following definition therefore makes sense.

\begin{definition}[Projection]\label{def:projection-to-simplex}
Let $A \subseteq \Delta$ be a simplex and let $C \subseteq \Delta$ be a chamber.
The \emph{projection of $C$ to $A$}, denoted by $\pr_A(C)$, is the unique chamber in $\st_{\Delta}(A)$ such that every minimal gallery from $A$ to $C$ starts with $\pr_A(C)$.
\end{definition}

In order to apply the observations on magic cubes we consider the following.

\begin{definition}\label{def:projection-vector}
Let $X \defeq \{1,\ldots,t\}$.
We consider the map
\[
\Phi \colon \Ch(\Delta) \rightarrow X^n,\ D \mapsto (x_1,\ldots,x_n),
\]
where $x_i$ is given by $\pr_{P_i}(D) = C_{i,x_i}$.
Further we define
\[
\mu \colon \mathcal{P}(X^n) \rightarrow \N_0,\ A \mapsto \abs{\Phi^{-1}(A)}
\]
and call $\mu(A)$ the \emph{weight} of $A$.
\end{definition}

\begin{lemma}\label{lem:distribution-of-preimages}
The map $\mu$ is an $n$-dimensional magic cube of weight $N \defeq \frac{\abs{\Ch(\Delta)}}{t}$ over $X$.
\end{lemma}
\begin{proof}
Let $P \subset \Delta$ be a panel and let $C_1,C_2$ be chambers in the star of $P$.
Since $\Delta$ is a thick Moufang building, there is a type-preserving automorphism $\alpha \in \Aut(\Delta)$ that fixes $P$ and satisfies $\alpha(C_1) = C_2$.
Let $D_1$ be a chamber in $\pr_P^{-1}(C_1)$ and let $\Gamma = E_1 | \cdots | E_m$ be a minimal gallery from $E_1 = C_1$ to $E_m = D_1$.
Note that by the definition of the projection, $\Gamma$ is a minimal gallery from $P$ to $D_1$.
By applying $\alpha$ to $\Gamma$, we therefore obtain a minimal gallery $\alpha(E_1) | \cdots | \alpha(E_m)$ from $\alpha(P) = P$ to $\alpha(E_m) = \alpha(D_1)$, that starts with $\alpha(E_1) = C_2$.
Thus we have $\pr_P(\alpha(D_1)) = C_2$ which shows that $\alpha(\pr_P^{-1}(C_1)) \subseteq \pr_P^{-1}(\alpha(C_1))$.
An application of $\alpha^{-1}$ reveals that in fact we have the equality $\alpha(\pr_P^{-1}(C_1)) = \pr_P^{-1}(\alpha(C_1))$.
In particular we see that every set of the form $\pr_P^{-1}(C)$, where $P \subset \Delta$ is a panel and $C \subset \st_{\Delta}(P)$ is a chamber, has the same cardinality.
Since the definition of projection ensures that there are no common chambers in $\pr_P^{-1}(C_1)$ and $\pr_P^{-1}(C_2)$ for $C_1 \neq C_2$, we obtain
\[
\abs{\pr_P^{-1}(C)} = \abs{\Ch(\Delta)} \cdot \abs{\Ch(\st_{\Delta}(P))}^{-1} = \abs{\Ch(\Delta)} \cdot t^{-1} = N.
\]
In the situation of the lemma we have $\abs{\pr_{P_i}^{-1}(C_{i,j})} = N$ for all $i,j$.
We can therefore apply Example~\ref{ex:magic-cube} to deduce that $\mu$ is an $n$-dimensional magic cube of weight $N$ over $X$.
\end{proof}

We can now apply Corollary~\ref{cor:zero-block} to the situation in Lemma~\ref{lem:distribution-of-preimages}.

\begin{proposition}\label{prop:independent-proj-panels}
Let $\Delta$ be a finite, uniformly thick Moufang building of thickness $t \defeq \thickness(\Delta)$ and let $P_1,\ldots,P_n$ be a sequence of distinct panels in $\Delta$.
Then we can find at least $m \defeq \ceil{\frac{t}{1+n^2}}$
chambers $D_1,\ldots,D_m$, such that
\[
\pr_{P_i}(D_{j_1}) \neq \pr_{P_i}(D_{j_2})
\]
for all $1 \leq j_1 \neq j_2 \leq m$ and $1 \leq i \leq n$.
\end{proposition}
\begin{proof}
Let $\{C_{i,1},\ldots,C_{i,t}\}$ be an enumeration of the chambers in $\st(P_i)$ and let $\mu \colon \mathcal{P}(X^n) \rightarrow \N_0$ be the corresponding magic cube from Lemma~\ref{lem:distribution-of-preimages}.
From Corollary~\ref{cor:zero-block} we know that there are permutations $\sigma_i \in \Sym(X)$ for $1 \leq i \leq n$ such that $(\sigma_1(i),\ldots,\sigma_n(i))$ has positive weight for all $1 \leq i \leq m$.
This means that we can find chambers $D_1,\ldots,D_m \in \Ch(\Delta)$, such that
\[
(\pr_{P_1}(D_j),\ldots,\pr_{P_n}(D_j)) = (C_{1,\sigma_1(j)},\ldots,C_{n,\sigma_n(j)})
\]
for every $1 \leq j \leq m$.
Since each $\sigma_i$ is a permutation, we obtain
\[
\pr_{P_i}(D_{j_1}) = C_{i,\sigma_i(j_1)} \neq C_{i,\sigma_i(j_2)} = \pr_{P_i}(D_{j_2})
\]
for all $1 \leq j_1 \neq j_2 \leq m$ and $1 \leq i \leq n$.
\end{proof}

\subsection{Intersections of convex hulls}

For the rest of this section we fix a finite, $d$-dimensional, uniformly thick, Moufang building $\Delta$ and a finite set
\[
\mathcal{E} \defeq \{E_1,\ldots,E_n\}
\]
of chambers in $\Delta$.
Let $t \defeq \thickness(\Delta)$ and let $S \defeq \{s_0,\ldots,s_{d}\}$ denote the set of types of panels in $\Delta$.
For every chamber $E \subseteq \Delta$ and every $0 \leq i \leq d$ let $P_{E,i}$ denote the panel of type $s_i$ in $E$.

\begin{lemma}\label{lem:independent-proj-chambers}
There are at least $m \defeq \ceil{\frac{t}{1+(d+1)^2n^2}}$ chambers $D_1,\ldots,D_m \subset \Delta$, such that
\[
\pr_{P_{E_j,i}}(D_k) \neq \pr_{P_{E_j,i}}(D_l)
\]
for all $0 \leq i \leq d$, $1 \leq j \leq n$ and $1 \leq k \neq l \leq m$.
\end{lemma}
\begin{proof}
Consider the set $\Set{P_{E_j,i}}{0 \leq i \leq d, 1 \leq j \leq n}$ of all panels incident to some chamber $E_j$.
Since the cardinality of this set is bounded above by $(d+1)n$, the lemma directly follows from Proposition~\ref{prop:independent-proj-panels}.
\end{proof}

In view of Lemma~\ref{lem:independent-proj-chambers} we can now fix a subset
\[
\mathcal{D} \subseteq \Ch(\Delta)
\]
of cardinality $m \defeq \ceil{\frac{t}{1+(d+1)^2n^2}}$ that satisfies $\pr_{P_{E_j,i}}(D) \neq \pr_{P_{E_j,i}}(D')$ for all $0 \leq i \leq d$, $1 \leq j \leq n$ and all distinct $D,D' \in \mathcal{D}$.

\begin{lemma}\label{lem:almost_disjoint_hulls_1}
Let $C,D_1,D_2$ be chambers in $\Delta$ with $\pr_{P_{C,i}}(D_1) \neq \pr_{P_{C,i}}(D_2)$ for every $0 \leq i \leq d$.
Then $\conv(C,D_1) \cap \conv(C,D_2) = C$.
\end{lemma}
\begin{proof}
Suppose first that there is a chamber $E \subseteq \conv(C,D_1) \cap \conv(C,D_2)$ with $E \neq C$.
In this case we know from Lemma~\ref{lem:conv-hull-of-two-chambers} that there are minimal galleries $\Gamma_1,\Gamma_2$ from $D_1$, respectively $D_2$, to $C$ that contain $E$.
Let
\[
\Gamma_1' = E\ \vert\ F_1\ \vert\ \cdots\ \vert\ F_k\ \vert\ C \text{ and } \Gamma_2' = E\ \vert\ G_1\ \vert\ \cdots\ \vert\ G_k\ \vert\ C
\]
be the subgalleries of $\Gamma_1$ and $\Gamma_2$ that start at $E$ and let $P$ the the panel between $F_k$ and $C$.
By replacing $\Gamma_2'$ with $\Gamma_1'$ in $\Gamma_2$ we get a minimal gallery of the form $\Gamma_2'' = D_2\ \vert\ \cdots\ \vert\ F_k\ \vert\ C$.
From this it follows that $\pr_P(D_1) = F_k = \pr_P(D_2)$, which contradicts our assumptions.
Thus, $C$ is the only chamber in $\conv(C,D_1) \cap \conv(C,D_2)$.

Now consider a point $x \in \conv(C,D_1) \cap \conv(C,D_2)$ and let $A \subseteq \conv(C,D_1) \cap \conv(C,D_2)$ be a cell containing $x$.
By (B1) there are apartments $\Sigma_1,\Sigma_2 \subseteq \Delta$ containing $C$ and $D_1$, respectively $C$ and $D_2$.
Since every apartment in a building is convex (see, e.g.,~\cite[4.40]{AbramenkoBrown08}), we have $\conv(C,D_1) \subseteq \Sigma_1$ and $\conv(C,D_2) \subseteq \Sigma_2$.
In this case~\cite[3.131]{AbramenkoBrown08} tells us that the chamber $\pr_A(C)$ lies in $\conv(C,D_1)$ and $\conv(C,D_2)$.
As observed above, this implies $\pr_A(C) = C$.
On the other hand, we have $x \in A \subseteq \pr_A(C)$, which proves the claim.
\end{proof}

\begin{corollary}\label{cor:independent-proj-chambers}
For every two distinct $D,D' \in \mathcal{D}$ we have
\[
\conv(E_j,D) \cap \conv(E_j,D') = E_j
\]
for every $1 \leq j \leq n$.
\end{corollary}
\begin{proof}
This is an immediate consequence of Lemmas~\ref{lem:independent-proj-chambers} and~\ref{lem:almost_disjoint_hulls_1}.
\end{proof}

Recall from Lemma~\ref{lem:polynomial-bounded} that the following definition makes sense.

\begin{definition}\label{def:max-vertices-apartment}
For each $d \in \N$ let $c_d \in \N$ denote the maximal number of cells that can be contained in an apartment of a $d$-dimensional finite thick building.
\end{definition}

\begin{lemma}\label{lem:disjoint-complexes-1}
Let $l \in \N$ and suppose that the thickness $t$ of $\Delta$ satisfies
\[
t > l \cdot (n^2 \cdot c_d)(1+(d+1)^2n^2).
\]
Then there are $l$ chambers $D_1,\ldots,D_l \in \mathcal{D}$ such that
\begin{equation}\label{eq:disjoint-complexes-1-1}
\left(\bigcup \limits_{i=1}^n \conv(E_i,D_{j_1})\right) \cap \conv(E_k,D_{j_2}) = E_k
\end{equation}
for all $1 \leq k \leq n$ and $1 \leq j_1 < j_2 \leq l$.
\end{lemma}
\begin{proof}
From Corollary~\ref{cor:independent-proj-chambers} we know that $\mathcal{D}_1 \defeq \mathcal{D}$ consists of $m = \ceil{\frac{t}{1+(d+1)^2n^2}}$ chambers such that
\begin{equation}\label{eq:intersection1}
\conv(E_j,D) \cap \conv(E_j,D') = E_j
\end{equation}
for every $1 \leq j \leq n$ and all distinct $D,D' \in \mathcal{D}_1$.
Let $D_1 \in \mathcal{D}_1$ be a fixed chamber.
We want to give an upper bound on the number of chambers $D \in \mathcal{D}_1$ that do not satisfy
\begin{equation}\label{eq:intersection2}
\left(\bigcup \limits_{i=1}^n \conv(E_i,D_1)\right) \cap \conv(E_j,D) = E_j
\end{equation}
for some given $1 \leq j \leq n$.
In view of~\eqref{eq:intersection1} we see that the number of such chambers $D$ is bounded by the number of vertices in $\bigcup \limits_{i=1}^n \conv(E_i,D_1)$.
Since the convex hull of two arbitrary chambers in $\Delta$ is contained in some apartment, the number of vertices in $\bigcup \limits_{i=1}^n \conv(E_i,D_1)$ is bounded above by $n \cdot c_d$.
Hence, the cardinality of the subset $\mathcal{K}_1 \subseteq \mathcal{D}_1$ that consists of chambers $D$ that do not satisfy~\eqref{eq:intersection2} for at least one $j$ is bounded above by $n^2 \cdot c_d$.
Let us now fix a chamber $D_2 \in \mathcal{D}_2 \defeq \mathcal{D}_1 \setminus (\mathcal{K}_1 \cup \{D_1\})$.
With the same argument as before we see that the subset $\mathcal{K}_2 \subseteq \mathcal{D}_2$ of chambers $D$ that do not satisfy
\begin{equation}\label{eq:intersection3}
\left(\bigcup \limits_{i=1}^n \conv(E_i,D_1) \cup \bigcup \limits_{i=1}^n \conv(E_i,D_2)\right) \cap \conv(E_j,D) = E_j
\end{equation}
for at least one $j$ consists of at most $n^2 \cdot c_d$ chambers.

Inductively, this procedure provides us with a sequence of subsets $\mathcal{D}_{i+1} \defeq \mathcal{D}_i \setminus (\mathcal{K}_i \cup \{D_i\})$ and chambers $D_i \in \mathcal{D}_i$ that satisfy the statement of the lemma, as long as $\mathcal{D}_i$ is large enough.
As each of the sets $\mathcal{K}_i$ has at most $n^2 \cdot c_d$ elements, this can be guaranteed if
\[
\abs{\mathcal{D}_{i+1}} \geq \abs{\mathcal{D}_1} - i(n^2 \cdot c_d-1) \geq m - i(n^2 \cdot c_d-1) > 0.
\]
From our assumption on $t$ we obtain
\[
m \geq \frac{t}{1+(d+1)^2n^2}
> \frac{l \cdot (n^2 \cdot c_d)(1+(d+1)^2n^2)}{1+(d+1)^2n^2}
> (l-1)(n^2 \cdot c_d-1).
\]
Thus we see that we can choose $l$ chambers $D_1,\ldots,D_l$ with the desired properties.
\end{proof}

As a direct consequence we obtain the following.

\begin{corollary}\label{cor:disjoint-complexes-2}
Suppose that the thickness $t$ of $\Delta$ satisfies
\[
t > l \cdot (n^2 \cdot c_d)(1+(d+1)^2n^2)
\]
for some $l \in \N$.
Then there are $l$ chambers $D_1,\ldots,D_l \in \mathcal{D}$ such that
\[
\left(\bigcup \limits_{i=1}^n \conv(E_i,D_{j_1})\right) \cap \left(\bigcup \limits_{i=1}^n \conv(E_k,D_{j_2})\right) = \bigcup \limits_{k=1}^n E_k
\]
for all $1 \leq j_1 < j_2 \leq l$.
\end{corollary}
\begin{proof}
Lemma~\ref{lem:disjoint-complexes-1} provides us with $l$ chambers $D_1,\ldots,D_l \subseteq \mathcal{D}$ that satisfy
\[
\left(\bigcup \limits_{i=1}^n \conv(E_i,D_{j_1})\right) \cap \conv(E_k,D_{j_2}) = E_k
\]
for all $1 \leq k \leq n$ and $1 \leq j_1 < j_2 \leq l$.
Thus we have
\begin{align*}
&\left(\bigcup \limits_{i=1}^n \conv(E_{i},D_{j_1})\right) \cap \left(\bigcup \limits_{k=1}^n \conv(E_{k},D_{j_2})\right)\\
= &\bigcup \limits_{k=1}^n \left(\bigcup \limits_{i=1}^n \conv(E_{i},D_{j_1}) \cap \conv(E_{k},D_{j_2})\right)\\
= &\bigcup \limits_{k=1}^n E_k,
\end{align*}
which proves the claim.
\end{proof}

\begin{theorem}\label{thm:disjoint-unions-of-apartments}
Let $\Delta$ be a finite, uniformly thick, $d$-dimensional Moufang building and let $\mathcal{E} \defeq \{E_1,\ldots,E_n\}$ be a set of chambers in $\Delta$.
Let $l \in \N$ and suppose that the thickness $t$ of $\Delta$ satisfies
\[
t > (n(d+1) + l) \cdot (n^2 \cdot c_d)(1+(d+1)^2n^2).
\]
Then there is an $l$-element subset $\{D_1,\ldots,D_l\} \subseteq \Opp_{\Delta}(\mathcal{E})$
such that
\[
\left(\bigcup \limits_{i=1}^n \conv(E_i,D_j)\right) \cap \left(\bigcup \limits_{i=1}^n \conv(E_i,D_k)\right) = \bigcup \limits_{i=1}^n E_i
\]
for all $1 \leq j \neq k \leq l$.
\end{theorem}
\begin{proof}
From Corollary~\ref{cor:disjoint-complexes-2} we know that there are $l' \defeq n(d+1) + l$ chambers $D_1,\ldots,D_{l'} \in \mathcal{D}$ that satisfy
\[
\left(\bigcup \limits_{i=1}^n \conv(E_i,D_{k_1})\right) \cap \left(\bigcup \limits_{i=1}^n \conv(E_i,D_{k_2})\right) = \bigcup \limits_{i=1}^n E_i
\]
for all $1 \leq k_1 \neq k_2 \leq l'$.
Recall that by the definition of $\mathcal{D}$ we further have
\[
\pr_{P_{E_j,i}}(D_{k_1}) \neq \pr_{P_{E_j,i}}(D_{k_2})
\]
for all $0 \leq i \leq d$, $1 \leq j \leq n$ and $1 \leq k_1 \neq k_2 \leq l'$.
This particularly tells us that for every $1 \leq j \leq n$ and every panel $P$ of $E_j$ there is at most one chamber $D \in \mathcal{D}$ with $\pr_P(D) = E_j$.
It is therefore sufficient to remove at most $n(d+1)$ chambers from $\{D_1,\ldots,D_{l'}\}$ in order to ensure that every chamber $D$ in the remaining subset $\widehat{\mathcal{D}} \subseteq \{D_1,\ldots,D_{l'}\}$ satisfies $\pr_P(D) \neq E_j$ for every $1 \leq j \leq n$ and every panel $P$ of $E_j$.
In this case it is an easy exercise (see, e.g.,~\cite[1.59(b)]{AbramenkoBrown08}) to show that every $D\in\widehat{\mathcal{D}}$ is opposite to each of the chambers $E_j$.
For the cardinality of $\widehat{\mathcal{D}}$ we have
\[
\abs{\widehat{\mathcal{D}}} \geq l' - n(d+1) = l,
\]
which proves the claim.
\end{proof}

\section{Some calculus}\label{sec:calc}

\begin{definition}\label{def:exponentially-decreasing}
We say that a function $f \colon \N \rightarrow \R$ is \emph{exponentially decreasing} if there is a constant $\delta \in [0,1)$ such that $f(n) < \delta^n$ for all but finitely many $n$.
\end{definition}

\begin{notation}\label{not:power-set}
Let $S$ be an arbitrary set.
As before, the set of all subsets of $S$ is denoted by $\mathcal{P}(S)$.
For every $n \in \N$ we further write $\mathcal{P}_{n}(S)$ to denote the set of subsets $A \subseteq S$ with $\abs{A} \leq n$.
\end{notation}

\begin{notation}\label{not:avoidance-number}
Let $S$ be a finite set and let $A_i \in \mathcal{P}(S),\ 1 \leq i \leq m$ be some pairwise disjoint subsets.
The set of subsets $T \subseteq S$ with $A_i \nsubseteq T$ for all $1 \leq i \leq m$ will be denoted by
\[
\mathcal{P}(S,\Set{A_i}{1 \leq i \leq m}).
\]
Given any $m,n \in \N$, we denote by $p_{m,n}(S)$ the maximal cardinality of a set of the form $\mathcal{P}(S,\mathfrak{A})$, where $\mathfrak{A} = \Set{A_i}{1 \leq i \leq m}$ consists of $m$ pairwise disjoint sets $A_i \in \mathcal{P}_n(S)$.
\end{notation}

\begin{lemma}\label{lem:independend-quotient}
Let $m,n \in \N$ and let $S$ be a finite set.
The quotient $\frac{p_{m,n}(S)}{\abs{\mathcal{P}(S)}}$ is bounded above by $\Big(\frac{2^n - 1}{2^n}\Big)^m$.
If $\abs{S} \geq m \cdot n$, then $\frac{p_{m,n}(S)}{\abs{\mathcal{P}(S)}} = \Big(\frac{2^n - 1}{2^n}\Big)^m$.
\end{lemma}
\begin{proof}
Let $T \subseteq S$, and let $A_i \in \mathcal{P}_n(S),\ 1 \leq i \leq m$ be some pairwise disjoint subsets.
For each $1 \leq i \leq m$ let $T_i \defeq T \cap A_i$.
Then $T$ is contained in $\mathcal{P}(S,\mathfrak{A})$ if and only if $T_i \neq A_i$ for all $1 \leq i \leq m$.
Note that this property does not depend on the cardinality of $T \cap \left(S \setminus \bigcup \limits_{i=1}^m A_i\right)$.
It therefore follows that
\begin{align*}
\frac{\abs{\mathcal{P}(S,\mathfrak{A})}}{\abs{\mathcal{P}(S)}}
&= \frac{\abs{\mathcal{P}(S \setminus \bigcup \limits_{i=1}^m A_i)} \cdot \prod \limits_{i=1}^m (\abs{\mathcal{P}(A_i)} - 1)}{\abs{\mathcal{P}(S)}}\\
&= \frac{2^{\abs{S \setminus \bigcup_{i=1}^m A_i}}}{2^{\abs{S}}} \cdot \prod \limits_{i=1}^m (2^{\abs{A_i}} - 1)\\
&= \frac{1}{2^{\abs{\bigcup_{i=1}^m A_i}}} \cdot \prod \limits_{i=1}^m (2^{\abs{A_i}} - 1)\\
&= \prod \limits_{i=1}^m \frac{(2^{\abs{A_i}} - 1)}{2^{\abs{A_i}}}.
\end{align*}
Since $\abs{A_i} \leq n$, it follows that $\frac{2^{\abs{A_i}} - 1}{2^{\abs{A_i}}} \leq \frac{2^n - 1}{2^n}$.
We conclude that $\frac{p_{m,n}(S)}{\abs{\mathcal{P}(S)}} \leq \Big(\frac{2^n - 1}{2^n}\Big)^m$.
Suppose now that $\abs{S} \geq m \cdot n$.
Then we can choose the subsets $A_i$ to have cardinality $\abs{A_i}=n$.
In this case the last statement of the lemma follows directly from the above chain of equalities.
\end{proof}

\begin{notation}\label{not:avoidance-probability}
Given any $m,n \in \N$, we write $p(m,n) \defeq \Big(\frac{2^n - 1}{2^n}\Big)^m.
$
\end{notation}

\begin{remark}
Note that by Lemma~\ref{lem:independend-quotient} we have
\[
p(m,n) = \max \limits_{S} \frac{p_{m,n}(S)}{\abs{\mathcal{P}(S)}},
\]
where $S$ runs over all finite sets.
\end{remark}

Recall that a function $f \colon \N \rightarrow \R_{\geq 0}$ is called \emph{subexponential} if $\lim \limits_{n \rightarrow \infty} \frac{f(n)}{a^n} = 0$ for every $a > 1$.

\begin{lemma}\label{lem:exponentially-decreasing}
Let $N \in \N$, $c \in \R$, and $\lambda > 0$ be some constants and let $\varepsilon \colon \N \rightarrow \R_{\geq 0}$ be a subexponential function.
Let $(t_n)_{n \in \N}$ be a sequence of natural numbers with $t_n \geq n$ for every $n \in \N$.
Then the function given by
\[
n \mapsto \varepsilon(t_n) \cdot p(\floor{\lambda t_n + c},N)
\]
is exponentially decreasing.
\end{lemma}
\begin{proof}
Let $\delta \defeq \frac{2^N - 1}{2^N}$.
From Lemma~\ref{lem:independend-quotient} it follows that $p(n,N) \leq \delta^n$ for every $n \in \N$.
Let $a > 1$ be such that $a \delta^{\lambda} < 1$.
Since $\varepsilon$ is subexponential we have $\lim \limits_{n \rightarrow \infty} \frac{\varepsilon(n)}{a^n} = 0$, which tells us that $\varepsilon(n) < a^n$ for all but finitely many $n \in \N$.
Thus we obtain
\[
\varepsilon(t_n) \cdot p(\floor{\lambda t_n + c},N)
< a^{t_n} \cdot \delta^{\floor{\lambda t_n + c}}
< a^{t_n} \cdot \delta^{\lambda t_n + c - 1}
= (a\delta^{\lambda})^{t_n} \cdot \delta^{c-1}
\]
for all but finitely many $n \in \N$.
Now the claim follows from our assumption that $t_n \geq n$ for every $n \in \N$.
\end{proof}

The main result of this section is the following technical observation, which will help us to prove some properties of random subcomplexes of finite buildings.

\begin{corollary}\label{cor:sequence-disjoint-sets}
Let $N \in \N$, $c \in \R$, and $\lambda > 0$ be some constants, let $(t_n)_{n \in \N}$ be a sequence of natural numbers with $t_n \geq n$, let $\varepsilon \colon \N \rightarrow \R_{\geq 0}$ be a subexponential function, and let $(X_n)_{n \in \N}$ be a sequence of finite sets.
Suppose that we can find non-empty subsets $\mathfrak{A}_n \subseteq \mathcal{P}(\mathcal{P}_{N}(X_n))$ of cardinality $\abs{\mathfrak{A}_n} \leq \varepsilon(t_n)$, such that every $\mathcal{A} \in \mathfrak{A}_n$ satisfies $\lambda \abs{\mathcal{A}} + c \geq t_n$, and $A \cap A' = \emptyset$ for all distinct $A,A' \in \mathcal{A}$.
Then the function
\[
h \colon \N \rightarrow \R_{\geq 0},\ n \mapsto \abs{\bigcup_{\mathcal{A} \in \mathfrak{A}_n} \mathcal{P}(X_n,\mathcal{A})} \cdot \abs{\mathcal{P}(X_n)}^{-1}
\]
is exponentially decreasing.
\end{corollary}
\begin{proof}
The function $h$ is bounded above by
\begin{align*}
\sum \limits_{\mathcal{A} \in \mathfrak{A}_n} \frac{\abs{\mathcal{P}(X_n,\mathcal{A})}}{\abs{\mathcal{P}(X_n)}}
&\leq \sum \limits_{\mathcal{A} \in \mathfrak{A}_n} p(\abs{\mathcal{A}},N)\\
&\leq \sum \limits_{\mathcal{A} \in \mathfrak{A}_n} p(\floor{\lambda^{-1}t_n - \lambda^{-1}c},N)\\
&\leq \varepsilon(t_n) \cdot p(\floor{\lambda^{-1}t_n - \lambda^{-1}c},N).
\end{align*}
Now it remains to apply Lemma~\ref{lem:exponentially-decreasing} which tells us that the function given by
\[
n \mapsto \varepsilon(t_n) \cdot p(\lambda^{-1}t_n - \lambda^{-1}c,N)
\]
is exponentially decreasing.
\end{proof}

\section{Random subcomplexes are chamber complexes}\label{sec:random-subcomplexes-are-chamber}

In this section we prove that most induced subcomplexes of uniformly thick, finite Moufang buildings are chamber complexes.
Recall that a finite-dimensional simplicial complex $X$ is a \emph{chamber complex} if
\begin{enumerate}
\item all maximal simplices have the same dimension and
\item any two maximal simplices $C,D$ can be connected by a \emph{gallery}, i.e., a sequence of maximal simplices $E_1,\ldots,E_n$ with $E_1 = C$, $E_n = D$ and $\dim(E_i \cap E_{i+1}) = \dim(X)-1$ for all $1 \leq i < n$.
\end{enumerate}

Maximal simplices in chamber complexes will be called \emph{chambers}.

\subsection{Dimension of maximal simplices}

Let us first take a look at the extent to which property (1) is inherited by random induced subcomplexes of finite Moufang buildings.
Note that there is a canonical one-to-one correspondence between induced subcomplexes of a simplicial complex $X$ and subsets of $X^{(0)}$.
We will therefore write $\mathcal{P}(X)$ to denote the set of induced subcomplexes of $X$.
Further we write $\mathcal{P}_n(X)$ to denote the set of induced subcomplexes of $X$ with at most $n$ vertices.

\begin{notation}\label{not:set-of-non-chamber-subcomplexes}
For a building $\Delta$ we define $\mathcal{P}_{\mathcal{L}}(\Delta) \subseteq \mathcal{P}(\Delta)$ to be the set of induced subcomplexes $X \subseteq \Delta$ that contain a maximal simplex $\sigma \subseteq X$ of dimension $\dim(\sigma) < \dim(\Delta)$.
Let $\mathcal{L}(\Delta) \subseteq \mathcal{P}(\Delta)$
denote the set of links $\lk_{\Delta}(\sigma)$ where $\sigma \subseteq \Delta$ is a simplex with $\dim(\sigma) < \dim(\Delta)$.
\end{notation}

\begin{proposition}\label{prop:random-subcomplexes-are-chamber-1}
Let $d \in \N$ and let $(\Delta_n)_{n \in \N}$ be a sequence of finite, uniformly thick, $d$-dimensional Moufang buildings.
Suppose that $t_n \defeq \thickness(\Delta_n) > n$ for every $n \in \N$.
Then the function
\begin{equation}\label{eq:random-subcomplexes-are-chamber-1}
g \colon \N \rightarrow \R_{\geq 0},\ n \mapsto \frac{\abs{\mathcal{P}_{\mathcal{L}}(\Delta_n)}}{\abs{\mathcal{P}(\Delta_n)}}
\end{equation}
is exponentially decreasing.
\end{proposition}
\begin{proof}
Note that if $X$ is an induced subcomplex of $\Delta_n$ that contains a maximal simplex $\sigma$ of dimension less than $d$, then $\lk_X(\sigma)$ is empty.
We therefore have
\[
\mathcal{P}_{\mathcal{L}}(\Delta_n) \subseteq \Set{X \in \mathcal{P}(\Delta_n)}{X \cap L = \emptyset, \text{ for some } L \in \mathcal{L}(\Delta_n)}.
\]
Let $\mathcal{Q}_n$ denote the latter set.
Let $\mathcal{A}_L \defeq \Set{\{x\} \in \mathcal{P}_1(\Delta_n)}{x \in L \text{ is a vertex}}$ for each $L \in \mathcal{L}(\Delta_n)$ and let $\mathfrak{A}_n \defeq \Set{\mathcal{A}_L}{L \in \mathcal{L}(\Delta_n)}$.
Note that with this terminology, the set $\mathcal{Q}_n$ can be written as
\[
\mathcal{Q}_n = \bigcup_{\mathcal{A}_L \in \mathfrak{A}_n} \mathcal{P}(\Delta_n,\mathcal{A}_L).
\]
From Lemma~\ref{lem:polynomial-bounded} we know that there is a polynomial $p_d$ with
\[
\abs{\mathfrak{A}_n} = \abs{\mathcal{L}(\Delta_n)} \leq p_d(t_n)
\]
for all $n \in \N$.
Furthermore we have $\abs{\mathcal{A}_L} = \abs{L} \geq t_n$ for every $\mathcal{A}_L \in \mathfrak{A}_n$, and polynomials are subexponential, so we can apply Corollary~\ref{cor:sequence-disjoint-sets} to deduce that the function $n \mapsto \frac{\abs{\mathcal{Q}_n}}{\abs{\mathcal{P}(\Delta_n)}}$
is exponentially decreasing.
Now the claim follows since $\abs{\mathcal{P}_{\mathcal{L}}(\Delta_n)} \leq \abs{\mathcal{Q}_n}$.
\end{proof}

\subsection{Existence of galleries}

\begin{definition}\label{def:k-chamber-contractible}
Let $n \in \N$, let $\Delta$ be a finite building, and let $X \subseteq \Delta$ be an induced subcomplex.
We say that $X$ satisfies the \emph{$n$-covering property} if for every $Y \in \mathcal{P}_n(X)$ there is a chamber $D_Y \in \Ch(\Delta)$ and some apartments $\Sigma_1, \ldots, \Sigma_m \subseteq \Delta$ such that
\begin{enumerate}[(a)]
\item $D_Y \subseteq \Sigma_i$ for $1 \leq i \leq m$,
\item $Y \subseteq \bigcup_{i=1}^{m} \op_{\Sigma_i}(D_Y)$, and
\item $\bigcup_{i=1}^{m} \Sigma_i \subseteq X$.
\end{enumerate}
Let $\mathcal{D}_n(\Delta)$ denote the set of induced subcomplexes $X \subseteq \Delta$ that do \emph{not} satisfy the $n$-covering property.
\end{definition}

For the sake of intuition, note for example that $X$ satisfies the $1$-covering property if and only if for every vertex $x$ of $X$ there is an apartment $\Sigma$ of $\Delta$ such that $x\in \Sigma\subseteq X$.

Recall from Definition~\ref{def:max-vertices-apartment} that for $d\in\N$, $c_d$ denotes the maximal number of cells that can be contained in an apartment of a $d$-dimensional finite thick building.

\begin{theorem}\label{thm:k-chamber-contractible}
Let $d,n \in \N$ and let $(\Delta_m)_{m \in \N}$ be a sequence of finite, uniformly thick, $d$-dimensional Moufang buildings.
Suppose that $t_m \defeq \thickness(\Delta_m)$ satisfies
\begin{equation}\label{eq:k-chamber-contractible}
t_m > (n(d+1) + m) \cdot (n^2 \cdot c_d)(1+(d+1)^2n^2)
\end{equation}
for every $m \in \N$.
Then the function
\[
g \colon \N \rightarrow \R_{\geq 0},\ m \mapsto \frac{\abs{\mathcal{D}_n(\Delta_m)}}{\abs{\mathcal{P}(\Delta_m)}}
\]
is exponentially decreasing.
\end{theorem}
\begin{proof}
For every $m \in \N$ let $\tau_m \in \N$ be maximal with
\[
t_m > (n(d+1) + \tau_m) \cdot (n^2 \cdot c_d)(1+(d+1)^2n^2).
\]
Then we can choose some constants $\lambda, c \in \N$ with $\lambda \tau_m + c \geq t_m$ for all $m \in \N$.
Let $I_m \defeq \mathcal{P}_n(\Ch(\Delta_m))$.
From Theorem~\ref{thm:disjoint-unions-of-apartments} we know that for every $\mathcal{E} \in I_m$, there are $\tau_m$ chambers $D_{\mathcal{E},1},\ldots,D_{\mathcal{E},\tau_m} \in \Opp_{\Delta_m}(\mathcal{E})$ such that the complexes
\[
B_{i,\mathcal{E}}^{(m)} \defeq \bigcup_{E \in \mathcal{E}} \conv(E,D_{\mathcal{E},i})
\]
satisfy
\[
B_{i,\mathcal{E}}^{(m)} \cap B_{j,\mathcal{E}}^{(m)} = \bigcup_{E \in \mathcal{E}} E
\]
for $1 \leq i \neq j \leq \tau_m$.
Let $A_{i,\mathcal{E}}^{(m)}$ denote the set of vertices in $B_{i,\mathcal{E}}^{(m)}$ that are not contained in $\bigcup \limits_{E \in \mathcal{E}} E$.
Note that $A_{1,\mathcal{E}}^{(m)},\ldots,A_{\tau_m,\mathcal{E}}^{(m)}$ are pairwise disjoint and that the cardinality of each $A_{i,\mathcal{E}}^{(m)}$ satisfies
\[
\abs{A_{i,\mathcal{E}}^{(m)}}
\leq \abs{\bigcup_{E \in \mathcal{E}} \conv(E,D_{\mathcal{E},i})}
\leq n c_d,
\]
which does not depend on $i,\mathcal{E}$ or $m$.
Let $N \defeq nc_d$.
In order to apply Corollary~\ref{cor:sequence-disjoint-sets}, we consider the sets
\[
\mathcal{A}_{\mathcal{E}}^{(m)} \defeq \Set{A_{i,\mathcal{E}}^{(m)}}{1 \leq i \leq \tau_m}  \subseteq \mathcal{P}_{N}(\Delta_m)
\]
for every $\mathcal{E} \in I_m$.
From Lemma~\ref{lem:polynomial-bounded} we know that
the number of cells in $\Delta_m$ is bounded above by a polynomial in $t_m$.
Since $\abs{I_m}$ is bounded above by a polynomial in $\abs{\Delta_m}$ it follows that $\abs{I_m}$ is a also bounded above by a polynomial in $t_m$. 
In particular we see that the cardinality of
\[
\mathfrak{A}^{(m)} \defeq \Set{\mathcal{A}_{\mathcal{E}}^{(m)}}{\mathcal{E} \in I_m} \subseteq \mathcal{P}(\mathcal{P}_{N}(\Delta_m))
\]
is bounded above by a subexponential function in $t_m$.
Recall that $\mathcal{P}(\Delta_m^{(0)},\mathcal{A}_{\mathcal{E}}^{(m)})$ denotes the set of subsets $Z \subseteq \Delta_m^{(0)}$ with $A_{i,\mathcal{E}}^{(m)} \nsubseteq Z$ for all $1 \leq i \leq \tau_m$.
Since
\[
\lambda \abs{\mathcal{A}_{\mathcal{E}}^{(m)}} + c = \lambda \tau_m + c \geq t_m
\]
for every $m \in \N$, we can now apply Corollary~\ref{cor:sequence-disjoint-sets} to deduce that the function
\begin{equation}\label{eq:k-chamber-contractible2}
h \colon \N \rightarrow \R_{\geq 0},\ m \mapsto \abs{\bigcup_{\mathcal{A}_{\mathcal{E}}^{(m)} \in \mathfrak{A}^{(m)}} \mathcal{P}(\Delta_m,\mathcal{A}_{\mathcal{E}}^{(m)})} \cdot \abs{\mathcal{P}(\Delta_m)}^{-1}
\end{equation}
is exponentially decreasing.

Consider now an induced subcomplex $X \subseteq \Delta_m$ that satisfies $X \notin \mathcal{P}(\Delta_m,\mathcal{A}_{\mathcal{E}}^{(m)})$ for every $\mathcal{A}_{\mathcal{E}}^{(m)} \in \mathfrak{A}^{(m)}$.
Suppose that every maximal simplex in $X$ is a chamber.
Then for every $Y \in \mathcal{P}_n(X)$ there is some $\mathcal{E} \in I_m$ with $Y \subseteq \bigcup \limits_{E \in \mathcal{E}} E \subseteq X$ and we can find an index $1 \leq i \leq \tau_m$ such that $A_{i,\mathcal{E}}^{(m)}$ is contained in the vertex set of $X$.
Together this tells us that all the vertices of $B_{i,\mathcal{E}}^{(m)}$ are contained in $X$.
As $X$ is an induced subcomplex of $\Delta_m$ it follows that the whole subcomplex $B_{i,\mathcal{E}}^{(m)} = \bigcup_{E \in \mathcal{E}} \conv(E,D_{\mathcal{E},i})$ lies in $X$.
Since $D_{\mathcal{E},i}$ lies in $\Opp_{\Delta_m}(\mathcal{E})$, each convex hull $\conv(E,D_{\mathcal{E},i})$ is an apartment $\Sigma_{E}$ in $\Delta$ with $\op_{\Sigma_E}(D_{\mathcal{E},i}) = E$.
Thus we see that $X$ satisfies the $n$-covering property.
It therefore follows that every $X \in \mathcal{D}_n(\Delta_m)$ either has maximal simplices that are not chambers, or lies in $\mathcal{P}(\Delta_m,\mathcal{A}_{\mathcal{E}}^{(m)})$ for some $\mathcal{A}_{\mathcal{E}}^{(m)} \in I_m$.
Since the sum of two exponentially decreasing functions is exponentially decreasing, the claim follows from~\eqref{eq:k-chamber-contractible2} together with~\eqref{eq:random-subcomplexes-are-chamber-1}.
\end{proof}

\begin{definition}\label{def:spherical}
A $d$-dimensional simplicial complex $X$ is called \emph{$d$-spherical}, or just \emph{spherical}, if $X$ is homotopy equivalent to a non-trivial wedge of $d$-spheres.
\end{definition}

\begin{theorem}[Most subcomplexes are chamber complexes]\label{thm:random-subcomplexes-of-arbitrary-sequences}
Let $d,k \in \N$ and let $(\Delta_n)_{n \in \N}$ be a sequence of finite, uniformly thick, $d$-dimensional Moufang buildings.
Let $\mathcal{A}_n \subseteq \mathcal{P}(\Delta_n)$ denote the subset consisting of non-empty subcomplexes $X \leq \Delta_n$ such that
\begin{enumerate}
\item $X$ is a union of apartments in $\Delta_n$,
\item $X$ is a chamber complex, and
\item Every subcomplex $Y \subseteq X$ with at most $k$ vertices is contained in a $d$-spherical subcomplex $Z \subseteq X$ that can be written as a union of at most $k$ apartments in $\Delta_n$.
\end{enumerate}
Suppose that there are some constants $\lambda, c > 0$ with $\thickness(\Delta_n) \geq \lambda \cdot n - c$ for all but finitely many $n \in \N$.
Then the function $n \mapsto \frac{\abs{\mathcal{P}(\Delta_n) \setminus \mathcal{A}_n}}{\abs{\mathcal{P}(\Delta_n)}}$ is exponentially decreasing.
\end{theorem}
\begin{proof}
Let $\mathcal{B}_n \subseteq \mathcal{P}(\Delta_n)$ denote the set of induced non-empty subcomplexes $X \leq \Delta_n$ such that
\begin{enumerate}[(a)]
\item $X$ satisfies the $((k+1) \cdot (d+1))$-covering property and
\item $X$ is a union of chambers of $\Delta_n$.
\end{enumerate}
Our first goal is to show that $\mathcal{B}_n \subseteq \mathcal{A}_n$ for every $n \in \mathbb{N}$.
Let $X \in \mathcal{B}_n$ and let $Y \leq X$ be a subcomplex with at most $k$ vertices.
Then $Y$ can be covered by a set $\mathcal{E} \subseteq \Ch(X)$ consisting of at most $k$ chambers.
As the number of vertices in $\bigcup_{E \in \mathcal{E}} E$ is bounded above by $k(d+1) < (k+1) \cdot (d+1)$, we can find apartments $\Sigma_E \subseteq X$ for $E \in \mathcal{E}$ and a chamber $D \subset \bigcap_{E \in \mathcal{E}} \Sigma_E$ such that $E = \op_{\Sigma_E}(D)$ for every $E \in \mathcal{E}$.
Since $Z \defeq \bigcup_{E \in \mathcal{E}} \Sigma_E$ is a union of apartments that contain a common chamber $D$, it follows that $Z$ is $d$-spherical.
To see this, we argue by induction on the cardinality of $\mathcal{E}$.
If $\abs{\mathcal{E}} = 1$ then the claim follows from the fact that an apartment in a spherical building is a triangulation of a $d$-dimensional sphere.
Suppose now that $Z' \defeq \bigcup_{E' \in \mathcal{E}'} \Sigma_{E'}$ is $d$-spherical for some proper, non-empty subset $\mathcal{E}' \subseteq \mathcal{E}$ and let $E \in \mathcal{E} \setminus \mathcal{E}'$.
Then $Z' \cap \Sigma_E =  \bigcup_{E' \in \mathcal{E}'} (\Sigma_{E'} \cap \Sigma_E)$ is a union of proper convex subcomplexes of $\Sigma_E$ containing $D$, each of which can be contracted by geodesics to a point $x \in D$.
Thus $Z' \cap \Sigma_E$ is contractible and it follows from a standard topological gluing lemma that $Z' \cup \Sigma_E$ is $(d-1)$-connected (see e.g.~\cite[10.3]{Bjoerner95}).
Since $Z' \cup \Sigma_E$ is moreover $d$-dimensional, it remains to recall that in this case $Z' \cup \Sigma_E = \bigcup_{E' \in \mathcal{E}' \cup \{E\}} \Sigma_{E'}$ is $d$-spherical (see e.g.~\cite[9.19]{Bjoerner95}).
Thus (3) is satisfied.
Note that (1) is a consequence of (3).
We want to show that $X$ satisfies (2).
To see this, let $C_1,C_2$ be two chambers in $X$.
The number of vertices in $C_1 \cup C_2$ is bounded above by $2(d+1) \leq (k+1) \cdot (d+1)$.
We can therefore find apartments $\Sigma_1,\Sigma_2 \subseteq X$ with $C_1 \subset \Sigma_1$ and $C_2 \subset \Sigma_2$ such that the intersection $\Sigma_1 \cap \Sigma_2$ contains a chamber $D$.
As apartments are gallery connected, there is a gallery $\Gamma_1$ from $C_1$ to $D$ that stays in $\Sigma_1$ and a gallery $\Gamma_2$ from $D$ to $C_2$ that stays in $\Sigma_2$.
By composing these two galleries, we obtain a gallery from $C_1$ to $C_2$ that is contained in $X$.
Thus $X$ is a chamber complex, which finally gives us $\mathcal{B}_n \subseteq \mathcal{A}_n$ for every $n \in \mathbb{N}$.

Now it suffices to prove that $n \mapsto \frac{\abs{\mathcal{P}(\Delta_n) \setminus \mathcal{B}_n}}{\abs{\mathcal{P}(\Delta_n)}}$ is exponentially decreasing.
From Proposition~\ref{prop:random-subcomplexes-are-chamber-1} and Theorem~\ref{thm:k-chamber-contractible} we know that we can choose constants $\lambda_0,c_0 > 0$ such that the function $n \mapsto \frac{\abs{\mathcal{P}(\Delta_n) \setminus \mathcal{B}_n}}{\abs{\mathcal{P}(\Delta_n)}}$ is exponentially decreasing if $\thickness(\Delta_n) > \lambda_0 \cdot n + c_0$ for all $n \in \N$.
Note that the property of being exponentially decreasing is invariant under precomposing with an affine function.
Thus our assumption that $\thickness(\Delta_n) \geq \lambda \cdot n - c$ for all but finitely many $n \in \N$ implies that $n \mapsto \frac{\abs{\mathcal{P}(\Delta_n) \setminus \mathcal{B}_n}}{\abs{\mathcal{P}(\Delta_n)}}$ is indeed exponentially decreasing.
\end{proof}

Recall that $\mathcal{F}_1(L)$ denotes the set of induced subcomplexes of $L$ that are not simply connected.

\begin{corollary}[Most subcomplexes are simply connected]\label{cor:trivial-fundamental-group}
Let $d \geq 2$ and let $(\Delta_n)_{n \in \N}$ be a sequence of finite, uniformly thick, $d$-dimensional Moufang buildings.
Suppose that there are some constants $\lambda, c > 0$ with $\thickness(\Delta_n) \geq \lambda \cdot n - c$ for all but finitely many $n \in \N$.
Then the function $n \mapsto \frac{\abs{\mathcal{F}_1(\Delta_n)}}{\abs{\mathcal{P}(\Delta_n)}}$ is exponentially decreasing.
\end{corollary}

\begin{proof}
Note that from the third point in Theorem~\ref{thm:random-subcomplexes-of-arbitrary-sequences} it not only follows that every complex $X \in \mathcal{A}_n$ is path-connected, but also that there is a uniform upper bound $L_d$ on the length of a shortest edge-path between two vertices in $X$.
We want to apply Theorem~\ref{thm:random-subcomplexes-of-arbitrary-sequences} for $k = 2L_d + 1$.
Being a simplicial complex, $X$ is simply connected if every cyclic edge path in $X$ is nullhomotopic.
Let $p = e_1e_2 \ldots e_{\ell}$ be such a cyclic edge path starting at $v_0$.
For each $1 \leq i < \ell$ let $v_i$ denote the terminal vertex of $e_i$ and let $q_i$ be a path from $v_0$ to $v_i$ of length at most $L_d$.
Up to homotopy we have
\[
p = e_1 \overline{q}_1 q_1 e_2 \overline{q}_2 q_2 \ldots e_{\ell-1} \overline{q}_{\ell-1} q_{\ell-1}  e_{\ell}
= (e_1 \overline{q}_1)
(q_1 e_2 \overline{q}_2) \cdots
(q_{\ell-2} e_{\ell-1} \overline{q}_{\ell-1}) (q_{\ell-1} e_{\ell}),
\]
where $\overline{q}_{i}$ denotes the inverse path of $q_i$.
If follows that $p$ can be written as a product of cyclic paths of the form $p_i = q_{i-1} e_{i} \overline{q}_{i}$ for $1 \leq i \leq \ell$ where we write $q_0$ and $q_{\ell}$ to denote the trivial path.
By construction each $p_i$ contains at most $k$ vertices, so point (3) of Theorem~\ref{thm:random-subcomplexes-of-arbitrary-sequences} tells us that $p_i$ is contained in a simply connected subcomplex of $X$.
Thus each $p_i$ represents the trivial element in $\pi_1(X,v_0)$ which proves the claim.
\end{proof}

\section{Buildings of type $A_n$}\label{sec:buildings-An}

In this section we restrict our attention to buildings of type $A_n$.
In this case we will be able to show some results on the higher connectivity properties of random subcomplexes of such buildings.
For the rest of this section we fix a sequence $(p_n)_{n \in \N}$ of ascending primes.
For $k\in\N$, consider the vector space $V_{k,n} \defeq \Field_{p_n}^{k+1}$ and the corresponding building $\Delta_{k,n} \defeq A(V_{k,n})$ of type $A_k$.

\begin{notation}\label{not:link-without-hyperplane}
Let $\mathcal{L}_{k,n} \subseteq \mathcal{P}(\Delta_{k,n})$ denote the set of induced subcomplexes $X \leq \Delta_{k,n}$ such that there is some $U \in X^{(0)}$ (i.e., some $0\ne U< V_{k,n}$) with $\dim(U) < k$ that is not contained in any hyperplane $H < V_{k,n}$ with $H \in X^{(0)}$.
\end{notation}

\begin{lemma}\label{lem:most-links-have-hyperplanes}
Let $k \in \N$.
The function
\[
n \mapsto \frac{\abs{\mathcal{L}_{k,n}}}{\abs{\mathcal{P}(\Delta_{k,n})}}
\]
is exponentially decreasing.
\end{lemma}
\begin{proof}
Let $\mathcal{H}_{k,n} \subseteq \Delta_{k,n}^{(0)}$ denote the subset of vertices that are given by hyperplanes of $V_{k,n}$.
Let $\mathcal{H}_{k,n}^c \defeq \Delta_{k,n}^{(0)} \setminus \mathcal{H}_{k,n}$ denote the complement of this set.
Note that for every $U \in \mathcal{H}_{k,n}^c$ there are at least $p_n+1 > n$ hyperplanes $H < V$ containing $U$.
Indeed, by fixing a $(k-1)$-dimensional subspace $U \leq W < V_{k,n}$, we see that the hyperplanes $H < V_{k,n}$ containing $W$, and hence $U$, correspond to the non-trivial, proper subspaces of $V/W \cong \mathbb{F}_{p_n}^2$.
Thus the number of hyperplanes $H < V_{k,n}$ containing $U$ is bounded below by the cardinality of the $1$-dimensional projective space over $\mathbb{F}_{p_n}$, which is $p_n+1$.
This tells us that the cardinality of $\mathcal{H}_{\Delta_{k,n}}(U) \defeq \lk_{\Delta_{k,n}}(U) \cap \mathcal{H}_{k,n}$ is bounded below by $n$.
In order to apply Corollary~\ref{cor:sequence-disjoint-sets} we consider the sets $\mathcal{A}_U \defeq \Set{\{H\} \in \mathcal{P}(\Delta_{k,n}^{(0)})}{H \in \mathcal{H}_{\Delta_{k,n}}(U)}$ and
$\mathfrak{A}_{k,n} \defeq \Set{\mathcal{A}_U}{U \in \mathcal{H}_{k,n}^c}$.
From Lemma~\ref{lem:polynomial-bounded} we know that $\abs{\mathcal{H}_{k,n}^c}$ is bounded above by a polynomial in $t_n \defeq \thickness(\Delta_{k,n}) = p_n + 1 > n$.
Since the elements of $\mathcal{A}_U$ are pairwise disjoint and have uniformly bounded cardinality, we can apply Corollary~\ref{cor:sequence-disjoint-sets} to deduce that
\[
h \colon \N \rightarrow \R_{\geq 0},\ n \mapsto \abs{\bigcup_{\mathcal{A}_U \in \mathfrak{A}_{k,n}} \mathcal{P}(\Delta_{k,n},\mathcal{A}_U)} \cdot \abs{\mathcal{P}(\Delta_{k,n})}^{-1}
\]
is an exponentially decreasing function.
Recall from Notation~\ref{not:avoidance-number} that $\mathcal{P}(\Delta_{k,n},\mathcal{A}_U)$ denotes the set of induced subcomplexes $X \leq \Delta_{k,n}$ with $A \nsubseteq X$ for every $A \in \mathcal{A}_U$.
Note that this just means that $X \cap \mathcal{H}_{\Delta_{k,n}}(U) = \emptyset$.
Since for every non-empty $X \in \mathcal{L}_{k,n}$ there is some $U \in \mathcal{H}_{k,n}^c$ with $X \cap \mathcal{H}_{\Delta_{k,n}}(U) = \emptyset$, it follows that $\mathcal{L}_{k,n}$ is contained in $\bigcup_{\mathcal{A}_U \in \mathfrak{A}_{k,n}} \mathcal{P}(\Delta_{k,n},\mathcal{A}_U)$.
Now the claim follows because $h$ is exponentially decreasing.
\end{proof}

\begin{notation}\label{not:intersection-hyperplanes}
Let $X \leq \Delta_{k,n}$ be an induced subcomplex.
For every hyperplane $H < V_{k,n}$ we define $X_H \defeq X \cap \st_{\Delta_{k,n}}(H)$.
More generally, if $\mathcal{H}$ is a set of hyperplanes in $V_{k,n}$, we write $X_{\mathcal{H}} \defeq \bigcap \limits_{H \in \mathcal{H}} X_H$.
\end{notation}

\begin{lemma}\label{lem:covering-with-hyperplanes}
Every complex $X \in \mathcal{P}(\Delta_{k,n}) \setminus \mathcal{L}_{k,n}$ can be covered by its subcomplexes $X_H$, where $H$ runs over the set of hyperplanes in $V_{k,n}$ with $H \in X^{(0)}$.
\end{lemma}
\begin{proof}
Let $\sigma \leq X$ be a simplex with vertex set $\{U_1,\ldots,U_m\}$.
Without loss of generality we may assume that $U_1 < \ldots < U_m$.
If $U_m$ is a hyperplane in $V_{k,n}$ then we are done.
Otherwise the assumption $X \notin \mathcal{L}_{k,n}$ tells us that there is a hyperplane $H < V_{k,n}$ with $H \in X^{(0)}$ that contains $U_m$.
In this case $H$ contains the other vertices $U_i$ as well, so the simplex $\tau$ with the vertex set $\{U_1,\ldots,U_m,H\}$ lies in $X$, which proves the claim.
\end{proof}

\begin{lemma}\label{lem:intersecting-stars}
Let $\mathcal{H}$ be a non-empty set of $m < k$ hyperplanes in $V_{k,n}$.
Then $(\Delta_{k,n})_{\mathcal{H}} \cong \Delta_{l,n} \ast \pt$ for some $k - m \leq l < k$, where $\pt$ denotes the graph consisting of a single vertex.
\end{lemma}
\begin{proof}
If a vertex $U \in \Delta_{k,n}$ is connected to a hyperplane $H$, then $U$ is a proper subspace of $H$.
By setting $V_{\mathcal{H}} \defeq \bigcap \limits_{H \in \mathcal{H}} H$, we see that the vertex set of $(\Delta_{k,n})_{\mathcal{H}}$ is given by $\Set{U < V_{k,n}}{0 < U \leq V_{\mathcal{H}}}$.
Thus we have $(\Delta_{k,n})_{\mathcal{H}} = A(V_{\mathcal{H}}) \ast \{V_{\mathcal{H}}\}$.
Note that the dimension of $V_{\mathcal{H}}$ is bounded from below by
\[
\dim(V_{\mathcal{H}}) \geq \dim(V_{k,n}) - \abs{\mathcal{H}} = k+1-m \geq 2.
\]
Thus we have
\[
(\Delta_{k,n})_{\mathcal{H}}
= A(V_{\mathcal{H}}) \ast \{V_{\mathcal{H}}\}
= \Delta_{\dim(V_{\mathcal{H}})-1,n} \ast \{V_{\mathcal{H}}\}
= \Delta_{l,n} \ast \{V_{\mathcal{H}}\}
\]
for some $l \geq k - m$.
\end{proof}

In order to prove our main theorem on random subcomplexes of buildings of type $A_n$, we will use the following version of the nerve lemma, which can be found in~\cite[Lemma 2.1]{Bjoerner94}.

\begin{lemma}\label{lem:nerve}
Let $X$ be a simplicial complex and let $\{Y_i\}_{i=1}^n$ be a family of subcomplexes such that $X = \bigcup_{i=1}^{n} Y_i$.
Suppose that there is some $k \in \N_0$ such that every non-empty intersection $Y_{i_1} \cap \ldots \cap Y_{i_m}$ is $(k-m+1)$-connected for each $m \geq 1$.
Then $X$ is $k$-connected if and only if $\mathcal{N}(\{Y_i\}_{i=1}^n)$, the nerve of the covering $\{Y_i\}_{i=1}^n$, is $k$-connected.
\end{lemma}

We know from Corollary~\ref{cor:trivial-fundamental-group} that for $k\ge 3$ (so $\dim(\Delta_{k,n})\ge 2$) and large enough $n$, ``most'' induced subcomplexes of $\Delta_{k,n}$ are simply connected.
Similarly, it follows from Theorem~\ref{thm:k-chamber-contractible} that for $k\ge 2$ and large enough $n$, ``most'' induced subcomplexes of $\Delta_{k,n}$ are connected.
Finally, it is obvious that for $k\ge 1$ and large enough $n$, ``most'' (i.e., all but one) induced subcomplexes of $\Delta_{k,n}$ are non-empty.
The following theorem concerns higher connectivity properties, and says that for any $k\ge 2$, ``most'' induced subcomplexes of $\Delta_{k,n}$ are $\floor{\frac{k-1}{2}}$-connected.
(Of course when $k=1$ we only get (and would only expect) $(-1)$-connected, not $0$-connected.)
We conjecture that for any $k\in\N$, ``most'' induced subcomplexes of $\Delta_{k,n}$ are $(k-2)$-connected; see Conjectures~\ref{conj:vanishing-homology} and~\ref{conj:vanishing-homology-general} for more precise statements.

Recall from the definition that $\mathcal{F}_m(L)$ denotes the set of induced subcomplexes $X \leq L$ such that $X$ fails to be $m$-connected.

\begin{theorem}[Most subcomplexes are highly connected]\label{thm:good-homology}
Let $k \in \N_{\geq 2}$.
The function
\[
\varphi_{k} \colon \N \rightarrow \R_{\geq 0},\ n \mapsto \frac{\abs{\mathcal{F}_{\floor{\frac{k-1}{2}}}(\Delta_{k,n})}}{\abs{\mathcal{P}(\Delta_{k,n})}}
\]
is exponentially decreasing.
\end{theorem}
\begin{proof}
Consider the subset $\mathcal{H}_{k,n} \subseteq \Delta_{k,n}^{(0)}$ of vertices that are given by hyperplanes of $V_{k,n}$.
We will prove the theorem by induction on $k$.
If $k = 2$, then the result follows from Theorem~\ref{thm:k-chamber-contractible}, and for $k = 3$ the result is covered by Corollary~\ref{cor:trivial-fundamental-group}.
Suppose now that the claim holds for some $k > 2$.
We have to prove that $\varphi_{k+1}$ is exponentially decreasing.
To see this, let us formulate sufficient conditions under which an induced subcomplex $X \leq \Delta_{k+1,n}$ is $\floor{\frac{k}{2}}$-connected.
From Lemma~\ref{lem:covering-with-hyperplanes} we know that every $X \in \mathcal{L}_{k,n}^c \defeq \mathcal{P}(\Delta_{k,n}) \setminus \mathcal{L}_{k,n}$ can be written as
\begin{equation}\label{eq:good-covering}
X = \bigcup \limits_{H \in I_X} X_H,
\end{equation}
where $I_X \defeq \mathcal{H}_{k+1,n} \cap X^{(0)}$ is the set of hyperplanes $H < V_{k+1,n}$ that represent a vertex of $X$.
Let $\mathcal{N}_X \defeq \mathcal{N}(\{X_H\}_{H \in I_X})$ be the nerve corresponding to the covering~\eqref{eq:good-covering}.
We want to apply the nerve lemma to deduce that $X$ is $\floor{\frac{k}{2}}$-connected if and only if $\mathcal{N}_X$ is $\floor{\frac{k}{2}}$-connected.
To do so, we have to check that every non-empty $X_{\mathcal{H}}$ with $\abs{\mathcal{H}} = m$ is $(\floor{\frac{k}{2}} - m + 1)$-connected for $m \geq 1$.
Note that for $m > \floor{\frac{k}{2}} + 2$ there is nothing to show.
For $m=1$ we see that $X_H = \st_X(H) = \lk_X(H) \ast \{H\}$ is contractible for $H \in I_X$.
Let us now study the case $2 \leq m \leq \floor{\frac{k}{2}} + 2 < k+1$, where the last inequality comes from our assumption that $k>2$.
In Lemma~\ref{lem:intersecting-stars} we saw that every $m$-element subset $\mathcal{H} \subseteq \mathcal{H}_{k+1,n}$ satisfies $(\Delta_{k+1,n})_{\mathcal{H}} \cong \Delta_{l,n} \ast \pt$ for some $l \in \N$ with
\[
1 \leq k+1 - \left( \floor{\frac{k}{2}} + 2 \right)
\leq k+1-m
\leq l
< k+1.
\]
In particular we have $l \leq k$, so that we can apply our induction hypothesis to deduce that
\[
\frac{\abs{\mathcal{F}_{\floor{\frac{l-1}{2}}}(\Delta_{l,n})}}{\abs{\mathcal{P}(\Delta_{l,n})}}
\]
is exponentially decreasing in $n$.
As $k+1-m \leq l$ and $m \geq 2$ we have
\[
\floor{\frac{k}{2}} - m + 1
= \floor{\frac{k - 2m + 2}{2}}
\leq \floor{\frac{k-m}{2}}
\leq \floor{\frac{l-1}{2}}.
\]
It therefore follows for $m \geq 2$ that
\[
\frac{\abs{\mathcal{F}_{\floor{\frac{k}{2}} - m + 1}(\Delta_{l,n})}}{\abs{\mathcal{P}(\Delta_{l,n})}}
\]
is exponentially decreasing in $n$.
Note that we have
\[
\frac{\abs{\mathcal{F}_{\floor{\frac{k}{2}} - m + 1}((\Delta_{k+1,n})_{\mathcal{H}})}}{\abs{\mathcal{P}((\Delta_{k+1,n})_{\mathcal{H}})}}
=
\frac{\abs{\mathcal{F}_{\floor{\frac{k}{2}} - m + 1}(\Delta_{l,n} \ast \pt)}}{\abs{\mathcal{P}(\Delta_{l,n} \ast \pt)}}
\leq
\frac{\abs{\mathcal{F}_{\floor{\frac{k}{2}} - m + 1}(\Delta_{l,n})}}{\abs{\mathcal{P}(\Delta_{l,n})}},
\]
so that $n \mapsto \frac{\abs{\mathcal{F}_{\floor{\frac{k}{2}} - m + 1}((\Delta_{k+1,n})_{\mathcal{H}})}}{\abs{\mathcal{P}((\Delta_{k+1,n})_{\mathcal{H}})}}$
is exponentially decreasing for every $m \geq 2$ as well.
For every $\mathcal{H} \in I_{k+1,n} \defeq \mathcal{P}_{\floor{\frac{k}{2}}+2}(\mathcal{H}_{k+1,n}) \setminus \mathcal{P}_{1}(\mathcal{H}_{k+1,n})$ of cardinality $m = \abs{\mathcal{H}}$ let $\mathcal{G}_{\mathcal{H}}$ denote the set of induced subcomplexes $X \leq \Delta_{k+1,n}$ for which $X_{\mathcal{H}}$ is not $(\floor{\frac{k}{2}} - m + 1)$-connected.
By definition we have $X_{\mathcal{H}} = X \cap (\Delta_{k+1,n})_{\mathcal{H}}$.
It therefore follows that
\begin{equation}\label{eq:good-covering-2}
\frac{\abs{\mathcal{G}_{\mathcal{H}}}}{\abs{\mathcal{P}(\Delta_{k+1,n})}}
=
\frac{\abs{\mathcal{F}_{\floor{\frac{k}{2}}-m+1}((\Delta_{k+1,n})_{\mathcal{H}})}}{\abs{\mathcal{P}((\Delta_{k+1,n})_{\mathcal{H}})}}.
\end{equation}
As noted above, the right-hand side of~\eqref{eq:good-covering-2} is exponentially decreasing in $n$.
On the other hand, it follows from Lemma~\ref{lem:polynomial-bounded} that $\abs{I_{k+1,n}}$ is bounded above by a polynomial in $n$.
By defining $\mathcal{G}_{k+1,n} \defeq \bigcup_{\mathcal{H} \in I_{k+1,n}} \mathcal{G}_{\mathcal{H}}$
we therefore see that
\[
n \mapsto
\frac{\abs{\mathcal{G}_{k+1,n}}}{\abs{\mathcal{P}(\Delta_{k+1,n})}}
\]
is exponentially decreasing in $n$.
By summarizing the above, we see that every $X \in \mathcal{P}(\Delta_{k+1,n}) \setminus (\mathcal{L}_{k+1,n} \cup \mathcal{G}_{k+1,n})$ can be written as a union
\begin{equation}
X = \bigcup \limits_{H \in I_X} X_H,
\end{equation}
such that for every non-empty subset $\mathcal{H} \subseteq \mathcal{H}_{k+1,n}$ of cardinality $1 \leq \abs{\mathcal{H}} \leq \floor{\frac{k}{2}} + 2$, the intersection
\[
X_{\mathcal{H}} \defeq \bigcap \limits_{H \in \mathcal{H}} X_H
\]
is $(\floor{\frac{k}{2}} - \abs{\mathcal{H}} + 1)$-connected.
The nerve lemma now tells us that $X$ is $\floor{\frac{k}{2}}$-connected if and only if $\mathcal{N}_X$ is $\floor{\frac{k}{2}}$-connected.
Moreover, the above shows that $X_{\mathcal{H}} \neq \emptyset$ for every $\mathcal{H}$ satisfying $\abs{\mathcal{H}} \leq \floor{\frac{k}{2}}+2$, so indeed $\mathcal{N}_X$ is $\floor{\frac{k}{2}}$-connected, and we are done.
\end{proof}

Recall that $\mathcal{T}_m(L)$ denotes the set of induced subcomplexes of $L$ that have trivial $m$th reduced homology.

\begin{theorem}\label{thm:non-vanishing-homology}
For every natural number $k$ we have $\lim \limits_{n \rightarrow \infty} \frac{\abs{\mathcal{T}_{k-1}(\Delta_{k,n})}}{\abs{\mathcal{P}(\Delta_{k,n})}} = 0$.
\end{theorem}
\begin{proof}
From Theorem~\ref{thm:random-subcomplexes-of-arbitrary-sequences} we know that there are subsets $\mathcal{A}_{k,n} \subseteq \mathcal{P}(\Delta_{k,n})$ with $\lim \limits_{n \rightarrow \infty} \frac{\abs{\mathcal{A}_{k,n}}}{\abs{\mathcal{P}(\Delta_{k,n})}} = 1$ such that every $X \in \mathcal{A}_{k,n}$ contains an apartment $\Sigma$ of $\Delta_{k,n}$.
Since $\Sigma$ is homeomorphic to a sphere of dimension $\dim(\Delta_{k,n}) = k-1$, it follows that $\widetilde{H}_{k-1}(X) \neq 0$.
Now the claim follows since $\mathcal{T}_{k-1}(\Delta_{k,n}) \subseteq \mathcal{P}(\Delta_{k,n}) \setminus \mathcal{A}_{k,n}$.
\end{proof}

\begin{definition}\label{def:spherical-subcomplexes}
Let $X$ be a $d$-dimensional simplicial complex.
We write $\mathcal{S}(X) \subseteq \mathcal{P}(X)$ to denote the set of $d$-dimensional, spherical subcomplexes of $X$.
\end{definition}

\begin{conjecture}\label{conj:vanishing-homology}
For every natural number $k$ we have $\lim \limits_{n \rightarrow \infty} \frac{\abs{\mathcal{S}(\Delta_{k,n})}}{\abs{\mathcal{P}(\Delta_{k,n})}} = 1$.
\end{conjecture}

More generally, we expect that Conjecture~\ref{conj:vanishing-homology} holds without any restrictions on the type of the (finite) building:

\begin{conjecture}\label{conj:vanishing-homology-general}
Let $d \in \N$ and let $(\Delta_n)_{n \in \N}$ be a sequence of finite $d$-dimensional buildings whose thickness satisfies $\thickness(\Delta_n) \rightarrow \infty$ as $n \rightarrow \infty$.
Then $\lim \limits_{n \rightarrow \infty} \frac{\abs{\mathcal{S}(\Delta_{n})}}{\abs{\mathcal{P}(\Delta_{n})}} = 1$.
\end{conjecture}

\section{Special case of graphs}\label{sec:graphs}

Recall that $\mathcal{P}(\Gamma)$ denotes the set of induced subgraphs of a graph $\Gamma$.
Also recall that $\mathcal{F}_0(\Gamma)$ denotes the set of induced subgraphs of $\Gamma$ that are not $0$-connected.

As far as we know, the following was unknown before.

\begin{theorem}\label{thm:square-free-graphs}
For every $\varepsilon>0$ there is a finite, bipartite, square-free graph $\Gamma$ such that
\[
\abs{\mathcal{F}_0(\Gamma)} < \varepsilon \abs{\mathcal{P}(\Gamma)}.
\]
\end{theorem}
\begin{proof}
Let $(p_n)_{n \in \N}$ be a sequence of increasing primes and let $\Gamma_n \defeq A_2(\Field_{p_n})$.
Recall that the vertices of $\Gamma_n$ are given by the non-trivial proper subspaces of $\Field_{p_n}^3$.
We can therefore partition the vertices into those that have dimension $1$ and those that have dimension $2$.
As there is no proper inclusion between vertices of the same dimension, we see that this partition provides a bipartite structure on $\Gamma_n$.
Next we suppose for a contradiction that $\Gamma_n$ contains a square.
Then we can find $4$ vertices $V_1,V_2,W_1,W_2 \in \Gamma_n^{(0)}$ such that $V_i$ is connected to $W_j$ for $i,j \in \{1,2\}$.
Without loss of generality we may assume that $\dim(V_i) = 1$ and $\dim(W_i) = 2$ for $i \in \{1,2\}$.
Note that in this case we have $V_1,V_2 \leq W_1 \cap W_2$, which is a contradiction since $V_1,V_2$ are distinct but $\dim(W_1 \cap W_2) = 1$.
It remains to prove the claim about the disconnected subgraphs.
As $\thickness(\Gamma_n) = p_n + 1 \rightarrow \infty$ for $n \rightarrow \infty$, we can apply Theorem~\ref{thm:good-homology} with $k=2$ to deduce that there is some $n \in \N$ with $\abs{\mathcal{F}_0(\Gamma_n)} < \varepsilon \abs{\mathcal{P}(\Gamma_n)}$.
\end{proof}

\section{Applications}\label{sec:apps}

In this section we combine the tools from Section~\ref{sec:vaf_RACG} with all the building theoretic results in the intervening sections to produce examples of RACGs with interesting virtual algebraic fibering properties.
As before, fix a sequence $(p_n)_{n \in \N}$ of ascending primes, let $V_{k,n} \defeq \Field_{p_n}^{k+1}$, and consider the corresponding buildings $\Delta_{k,n} \defeq A(V_{k,n})$ of type $A_k$.

\begin{theorem}\label{thrm:typeA_higher_fibering}
Let $k\in\N_{\ge 2}$. For all but finitely many $n$, the group $W_{\Delta_{k,n}}'$ algebraically $\F_{\floor{\frac{k+1}{2}}}$-fibers.
\end{theorem}

\begin{proof}
The chromatic number of $\Delta_{k,n}$ is $k$ for all $n$, since the vertices are the proper non-trivial subspaces of $\Field_{p_n}^{k+1}$, and so ``dimension'' is a minimal coloring by $k$ colors. Hence the result is immediate from Corollary~\ref{cor:almost_all_fiber} and Theorem~\ref{thm:good-homology}.
\end{proof}

Following Conjecture~\ref{conj:vanishing-homology}, we expect that Theorem~\ref{thrm:typeA_higher_fibering} can be improved to say that $W_{\Delta_{k,n}}'$ algebraically $\F_{k-1}$-fibers.
More generally, following Conjecture~\ref{conj:vanishing-homology-general}, we expect for any sufficiently thick, $d$-dimensional finite building $\Delta$ that $W_\Delta$ algebraically $\F_d$-fibers.

We can also deduce some results involving negative finiteness properties.

\begin{theorem}\label{thrm:low_dim_sharp_fiber}
Let $k\in\N$ with $k\le 3$.
For all but finitely many $n$, the group $W_{\Delta_{k,n}}'$ admits a map to $\Z$ whose kernel is of type $\F_{k-1}$ but not of type $\FP_k$.
\end{theorem}

\begin{proof}
If $k=1$ then $\Delta_{k,n}$ is a discrete set of at least $3$ vertices, so $W_{\Delta_{1,n}}'$ is a non-abelian free group, and the result holds.
Now assume $k\in\{2,3\}$.
The dimension of $\Delta_{k,n}$ is $k-1$ and the chromatic number is $k$, so by Corollary~\ref{cor:almost_all_sharply_fiber} it suffices to show that the functions
\[
n \mapsto \frac{\abs{\mathcal{F}_{k-2}(\Delta_{k,n})}}{\abs{\mathcal{P}(\Delta_{k,n})}} \text{ and } n \mapsto \frac{\abs{\mathcal{T}_{k-1}(\Delta_{k,n})}}{\abs{\mathcal{P}(\Delta_{k,n})}}
\]
both go to $0$ as $n$ goes to $\infty$.
Note that $k-2 = \floor{\frac{k-1}{2}}$ for $k=2$ and $k=3$.
Thus $\frac{\abs{\mathcal{F}_{k-2}(\Delta_{k,n})}}{\abs{\mathcal{P}(\Delta_{k,n})}}$ is exponentially decreasing in $n$ by Theorem~\ref{thm:good-homology}.
This handles the first function.
For the second function (for either $k=2$ or $3$), by Theorem~\ref{thm:non-vanishing-homology} we have that $\frac{\abs{\mathcal{T}_{k-1}(\Delta_{k,n})}}{\abs{\mathcal{P}(\Delta_{k,n})}}$ goes to $0$ with $n$, so we are done.
\end{proof}

Note that, following Conjecture~\ref{conj:vanishing-homology-general} and using Lemma~\ref{lem:force_sharply_legal_system}, we expect the following should be true:

\begin{conjecture}\label{conj:buildings_sharp_fibering}
Let $d \in \N$ and let $(\Delta_n)_{n \in \N}$ be a sequence of finite $d$-dimensional buildings whose thickness satisfies $\thickness(\Delta_n) \rightarrow \infty$ as $n \rightarrow \infty$.
Then for all but finitely many $n$, there is a kernel of a map $W_{\Delta_n}'\to\Z$ that is of type $\F_d$ but not $\FP_{d+1}$.
\end{conjecture}

Let us now focus on hyperbolic RACGs, i.e., those $W_L$ for which $L$ is square-free.
For $L$ of the form $\Delta_{k,n}$, this is equivalent to saying that $k\le 2$.
For $k\le 1$, $\Delta_{k,n}$ has no edges, and cannot have exactly two vertices, so $W_{\Delta_{k,n}}$ is virtually free but not isomorphic to $D_\infty$, and so does not virtually algebraically fiber.
Hence we are most interested in $k=2$.
In this case, we have the following:

\begin{theorem}\label{thrm:typeA_graph_fibering}
For all but finitely many $n$, the group $W_{\Delta_{2,n}}'$ is hyperbolic and algebraically fibers, via a map $W_{\Delta_{2,n}}'\to\Z$ whose (finitely generated) kernel is not hyperbolic.
\end{theorem}

\begin{proof}
By the proof of Theorem~\ref{thm:square-free-graphs}, the graph $\Delta_{2,n}$ is square-free, so $W_{\Delta_{2,n}}$ and hence $W_{\Delta_{2,n}}'$ is hyperbolic by Citation~\ref{cit:hyp_RACG}.
By Theorem~\ref{thrm:typeA_higher_fibering} (and as mentioned in Example~\ref{ex:bipartite}), $W_{\Delta_{2,n}}'$ algebraically fibers, and by Theorem~\ref{thrm:low_dim_sharp_fiber} the relevant kernels are not of type $\FP_2$, hence not hyperbolic.
\end{proof}

As a remark, we believe that $\Delta_{2,n}$ satisfies the hypotheses of \cite[Lemma~5.1]{JankiewiczNorinWise21} for large enough $n$, which also implies that $W_{\Delta_{2,n}}'$ algebraically fibers, but we are not sure whether the non-hyperbolicity of the relevant kernels also follows from tools in \cite{JankiewiczNorinWise21}.
It is also worth pointing out that since $W_{\Delta_{2,n}}'$ has cohomological dimension $2$ (by virtue of $X_{\Delta_{2,n}}$ having dimension $2$ and $W_{\Delta_{2,n}}'$ not being free), a result of Gersten \cite{Gersten96} says that the relevant kernels failing to be hyperbolic is precisely equivalent to them failing to be of type $\FP_2$.
Hence using Theorem~\ref{thrm:low_dim_sharp_fiber} to get them to be non-hyperbolic by virtue of them being not of type $\FP_2$ was a precise characterization, and not ``overkill''.
As another remark, it is clear that the requirement that the thickness be large enough is necessary, since when $L$ is a thin building of type $A_2$, that is a hexagon, we know that $W_L$ does not virtually algebraically fiber (Lemma~\ref{lem:planar_no_fiber}).

\bibliographystyle{alpha}

\end{document}